\documentclass[10pt,leqno]{article}

\usepackage{amsmath,amssymb,amsthm,mathrsfs,dsfont}

\usepackage[margin=3cm]{geometry} 

\usepackage{titlesec,hyperref}

\usepackage{color}

\usepackage{fancyhdr}
\titleformat{\subsection}{\it}{\thesubsection.\enspace}{1pt}{}

\newtheorem{theo}{Theorem}[section]
\newtheorem{lemm}[theo]{Lemma}

\newtheorem{coro}[theo]{Corollary}
\newtheorem{prop}[theo]{Proposition}
\newtheorem{rema}[theo]{Remark}
\numberwithin{equation}{section}

\allowdisplaybreaks 

\begin{document}
\title{Global regularity and large time behavior for some inviscid Oldroyd-B models in $\mathbb{R}^2$
\hspace{-4mm}
}

\author{ Wenjie $\mbox{Deng}^1$ \footnote{E-mail: detective2028@qq.com},\quad
Zhaonan $\mbox{Luo}^1$\footnote{E-mail:  1411919168@qq.com} \quad and\quad
Zhaoyang $\mbox{Yin}^{1,2}$\footnote{E-mail: mcsyzy@mail.sysu.edu.cn}\\
$^1\mbox{Department}$ of Mathematics,
Sun Yat-sen University, Guangzhou 510275, China\\
$^2\mbox{Faculty}$ of Information Technology,\\ Macau University of Science and Technology, Macau, China}

\date{}
\maketitle
\hrule

\begin{abstract}
In this paper, we are concerned with global strong solutions and large time behavior for some inviscid Oldroyd-B models. We first establish the energy estimate and B-K-M criterion for the 2-D co-rotation inviscid Oldroyd-B model. Then, we obtain global strong solutions with large data in Sobolev space by proving the boundedness of vorticity. As a corollary, we prove global existence of the corresponding inviscid Hooke model near equilibrium. Furthermore, we present global existence for the 2-D co-rotation inviscid Oldroyd-B model in critical Besov space by a refined estimate in Besov spaces with index $0$. Finally, we study large time behaviour for the noncorotation inviscid Oldroyd-B model. Applying the Fourier splitting method, we prove the $H^1$ decay rate for global strong solutions constructed by T. M. Elgindi and F. Rousset in \cite{2015Elgindi}. \\
\vspace*{5pt}

\vspace*{5pt}
\end{abstract}

\vspace*{10pt}

\tableofcontents

\section{Introduction}
We focus here upon the general Oldroyd-B models \cite{2015Elgindi} for visco-elastic flow:
\begin{align}\label{eq0}
\left\{\begin{array}{l}
\partial_tu+u\cdot\nabla u+\nabla P ={\rm div}~\tau+\nu\Delta u,~~~~{\rm div}~u=0,\\[1ex]
\partial_t\tau+u\cdot\nabla\tau+a\tau+Q(\nabla u,\tau)=\alpha D(u)+\mu\Delta\tau,\\[1ex]
u|_{t=0}=u_0,~~\tau|_{t=0}=\tau_0. \\[1ex]
\end{array}\right.
\end{align}
In \eqref{eq0}, $u(t,x)$ denotes the velocity of the polymeric liquid, $\tau(t,x)$ represents the symmetric tensor of constrains and $P$ is the pressure. The parameters $a$, $\mu$ and $\nu$ are nonnegative and $\alpha>0$.
Moreover,
$$Q(\nabla u, \tau)=\tau \Omega-\Omega\tau+b(D(u)\tau+\tau D(u)),$$
with $b\in[-1, 1]$, the vorticity tensor $\Omega=\frac {\nabla u-(\nabla u)^T} {2}$ and the deformation tensor $D(u)=\frac {\nabla u+(\nabla u)^T} {2}$.
For more explanations on the modeling, one can refer to \cite{1958Non}.

Taking $b=1$ and $\alpha=2$, then the general Oldroyd-B models \eqref{eq0} can be derived from the the following micro-macro models \cite{Doi1988,Zhang2007} with the Hooke potential $\mathcal{U}=\frac 1 2 |q|^2$, $\int_{\mathbb{R}^d}\psi dq =\int_{\mathbb{R}^d}\psi_0 dq= 1$ and the drag term $\sigma(u)=\nabla u$:
\begin{align}\label{eq1}
\left\{
\begin{array}{ll}
\partial_tu+u\cdot\nabla u+\nabla P = {\rm div}~\tau+\nu\Delta u,~~~~{\rm div}~u=0,\\[1ex]
\psi_t+u\cdot\nabla\psi={\rm div}_{q}~[- \sigma(u)\cdot{q}\psi+\frac a 2\nabla_{q}\psi+\frac a 2\nabla_{q}\mathcal{U}\psi]+\mu\Delta\psi, \\[1ex]
\tau_{ij}=\int_{\mathbb{R}^{d}}(q_{i}\nabla_{q_j}\mathcal{U})\psi dq-Id, \\[1ex]
u|_{t=0}=u_0,~~\psi|_{t=0}=\psi_0. \\[1ex]
\end{array}
\right.
\end{align}
In \eqref{eq1}, the polymer particles are described by the distribution function $\psi(t,x,q)$. Here the polymer elongation $q$ satisfies $q\in\mathbb{R}^d$, which means that the extensibility of the polymers is infinite and $x\in\mathbb{R}^d$.
$\tau$ is an extra-stress tensor which generated by the polymer particles effect. In general, $\sigma(u)=\nabla u$. For the co-rotation case, $\sigma(u)=\Omega$.

The micro-macro models are of great interest in many branches of physics, chemistry, and biology, which describe the system coupling fluids and polymers. In the models, a polymer is idealized as an "elastic dumbbell" consisting of two "beads" joined by a spring that can be modeled by a vector $q$. At the level of liquid, the system couples the Navier-Stokes equation or the Euler equation for the fluid velocity with a Fokker-Planck equation describing the evolution of the polymer density. (For more details, one can refer to $\cite{Masmoudi2013}$).

When $\int_{\mathbb{R}^d}\psi_0 dq = 1$, the following co-rotation inviscid Oldroyd-B model can be derived from the micro-macro model \eqref{eq1} with $\nu=0$, $\mathcal{U}=\frac 1 2 |q|^2$ and $\sigma(u) = \Omega$:
\begin{align}\label{eq2}
\left\{\begin{array}{l}
\partial_tu+u\cdot\nabla u+\nabla P={\rm div}~\tau,~~~~{\rm div}~u=0,\\[1ex]
\partial_t\tau + u\cdot\nabla\tau+a\tau+Q(\Omega,\tau)=\mu\Delta\tau,\\[1ex]
u|_{t=0}=u_0,~~\tau|_{t=0}=\tau_0,
\end{array}\right.
\end{align}
where $Q$ is the following bilinear form $Q(\Omega, \tau)=\tau \Omega(u)-\Omega(u)\tau$.	

Notice that the equations \eqref{eq2} reduces to the well-known Euler equation by taking $\tau=0$. However, taking $\tau=0$ in \eqref{eq0} with $\nu=0$, then we have $D(u)=0$, which implies $u=0$ in Sobolev spaces. The observation reveals the essential difference between \eqref{eq0} and \eqref{eq2}.

\subsection{The general Oldroyd-B models}
T. M. Elgindi and F. Rousset \cite{2015Elgindi} first proved global regularity for the 2-D Oldroyd-B type models \eqref{eq0} with $\nu=0$. Later on, T. M. Elgindi and J. Liu \cite{2015Elgindi1} obtained global strong solutions of the 3-D case under the assumption that initial data is sufficiently small. For the case $a=0$, P. Constantin, J. Wu, J. Zhao and Y. Zhu \cite{P.Constantin} established the global well-posedness of the inviscid Oldroyd-B models for fractional dissipation with small data.

Taking $\nu>0$ and $\mu=0$ in \eqref{eq0}, we obtain the classical Oldroyd-B model. In \cite{Guillope1990}, C. Guillop\'e,  and J. C. Saut first showed that the Oldroyd-B model admits a unique global strong solution in Sobolev spaces. The $L^p$-setting was given by E. Fern\'andez-Cara, F.Guill\'en and R. Ortega \cite{Fernandez-Cara}. The weak solutions of the Oldroyd-B model were proved by P. L. Lions and N. Masmoudi \cite{Lions-Masmoudi} for the case $b=0$. Notice that the problem for the case $b\neq0$ is still open, see \cite{2011Global,Masmoudi2013}. Later on, J. Y. Chemin and N. Masmoudi \cite{Chemin2001}
proved the existence and uniqueness of strong solutions in homogenous Besov spaces with critical index of regularity. Optimal decay rates for solutions to the 3-D Oldroyd-B model were obtained by M. Hieber, H. Wen and R. Zi \cite{2019OldroydB}. An approach based on the deformation tensor can be found in \cite{Li2,Li1,Lei2008,Lei-Zhou2005,2010On,2010On1,Zhang-Fang2012}.

\subsection{The Hooke models}
 Let $\nu,\mu>0$. The construction of global weak solutions for micro-macro systems was considered in \cite{Hookeweak1,Hookeweak2,Hookeweak3,Hookeweak4,Hookeweak5,Hookeweak6}. Recently, global existence and uniqueness of a large class of initial data for the diffusive 2-D models was proved in \cite{Diff2DF-P}. It's worthy mentioning that the so-called moments $(u,M_{a,b})$ considered in \cite{Diff2DF-P} are strong solutions with macroscopic variables $(t,x)$ while $\psi$ is nonnegative measures on $\mathbb{R}^2_q\times\mathbb{R}^2_x$ merely.

 Let $\nu>0,~\mu=0$. The local existence of micro-macro systems were proved by many researchers in different settings, see \cite{2004Well,Renardy1989An}. F. Lin, C. Liu and P. Zhang \cite{Zhang2007} studied the incompressible micro-macro polymeric system and proved global existence near equilibrium with some assumptions on the potential $\mathcal{U}$. Global regularity for the 2-D co-rotation Hooke dumbbell model was proved by N. Masmoudi, P. Zhang, and Z. Zhang \cite{HookeGlobal}. Long time behavior for the 3-D micro-macro polymeric system was considered by L. He and P. Zhang \cite{He2009}.

\subsection{Main results}
Our main results can be stated as follows.
\begin{theo}[Global well-posedness in Sobolev space]\label{th1}
Let $d=2~and~s>2$. Assume that $a>0$, $\mu>0$ and $\kappa=\min\{a,\mu\}$. Let $(u,\tau)$ be a strong solution of \eqref{eq2} with the initial data $(u_0,\tau_0)\in H^s$. Then there exists some sufficiently small constant $c$, which is not dependent on $a$, $\mu$ and the initial data, such that if
\begin{align}\label{condition1}
\|\nabla u_0\|_{L^2} \leq c\kappa,~~\|\tau_0\|_{H^1} \leq c(a\mu)^{\frac 1 2}\kappa,
\end{align}
and
\begin{align}\label{condition2}
\|\Gamma_0\|_{L^\infty}\leq ca\mu,~~~~\|\tau_0\|_{H^1}\leq \frac {c^2\lambda}{\ln(e+\|(u_0,\tau_0)\|^2_{H^s})},
\end{align}
where $\lambda=\min\{a^{\frac 1 2}\mu^{\frac 3 2},a^2\mu^2,a^2\mu,a,a^{\frac 3 2}\mu^{\frac 5 2},a^{\frac 1 2}\mu^{\frac 1 2}\}$, then the system \eqref{eq2} admits a unique global strong solution $(u,\tau)\in C([0,\infty); H^s)$.
\end{theo}
\begin{rema}
Let $\phi_0(x)=A(x_2e^{-|x|^2},-x_1e^{-|x|^2})^{T}$ and $\varphi_0(x)=Ae^{-|x|^2}Id$. Consider $u_0=\varepsilon\phi_0(\varepsilon x)$ and $\tau_0=\varepsilon^2\varphi_0(\varepsilon x)$, then we can verify that ${\rm div}~u_0=0$. We infer that $\|u_0\|_{L^2}=\|\phi_0\|_{L^2}$ and $\|\tau_0\|_{L^2}=\varepsilon\|\varphi_0\|_{L^2}$. Moreover, we deduce that $\|u_0\|_{\dot{H}^s}=\varepsilon^s\|\phi_0\|_{\dot{H}^s}$ and $\|\tau_0\|_{\dot{H}^s}=\varepsilon^{s+1}\|\varphi_0\|_{\dot{H}^s}$ for any $\varepsilon>0$. Finally, we can construct large initial data in $H^s$ which satisfies \eqref{condition1} and \eqref{condition2} by taking $A$ sufficiently large and $\varepsilon$ small enough.
\end{rema}
\begin{rema}
For any $a$ and $\mu$, the system \eqref{eq2} reduces to the well-known Euler equation by taking $\tau=0$. In this case, the parameters $a$ and $\mu$ in Theorem \ref{th1} can be regarded as infinity, which means that our results cover the global existence for the 2-D Euler equation in Sobolev spaces $H^s$.
\end{rema}
\begin{rema}
Notice that equations \eqref{eq2} contain more solutions than equations \eqref{eq1}. In the Corollary \ref{th1'}, we establish the connection between the solutions $(u,\tau)$ of \eqref{eq2} constructed in Theorem \ref{th1} and the solutions $(u,\psi)$ of \eqref{eq1}.
\end{rema}
\begin{theo}[Global well-posedness in critical Besov space]\label{th13}
	Let $d=2$. Assume that $a>0$ and $\mu>0$. Let $(u,\tau)$ be a strong solution of \eqref{eq2} with the initial data $(u_0,\tau_0)\in (H^1\cap B^1_{\infty,1})\times (H^1\cap B^0_{\infty,1})$. Let $H_0=\|(u_0,\tau_0)\|^2_{H^1} e^{\frac{C}{a\mu}+\frac{C}{a\mu}\|\tau_0\|^2_{L^{\infty}}+C\mu^{-2}\|\tau_0\|^2_{L^2}}$. Then there exists some sufficiently small constant $c$ , which is not dependent on $a$, $\mu$ and the initial data, such that if
	\begin{align}\label{highsmallness}
		\|(\mu\nabla u_0,\tau_0)\|_{B^0_{\infty,1}} \leq \frac{c\min\{a\mu,a^2\mu\}}{1+H^{\frac 12}_0},
	\end{align}
	and
	\begin{align}\label{nonlinearsmallness}
	\|\tau_0\|_{B^0_{\infty,1}\cap L^4}
	\leq \frac {c^2\beta}{1+H^{\frac 32}_0},
	\end{align}
	where $\beta=\min\{a^2\mu^2,a\mu^3,a^3\mu,a^3\mu^2,a^{\frac 52}\mu^{\frac32},a\mu,a^{\frac 1 2}\mu^{\frac 3 2},a^2\}$, then the system \eqref{eq2} admits a global strong solution $(u,\tau)\in C\big([0,\infty);H^1\cap B^1_{\infty,1}\big)\times C\big([0,\infty);H^1\cap B^0_{\infty,1}\big)$.
\end{theo}
\begin{rema}
For any $a$ and $\mu$, the system \eqref{eq2} reduces to the well-known Euler equation by taking $\tau=0$. In this case, the parameters $a$ and $\mu$ in Theorem \ref{th13} can be regarded as infinity, which means that our results cover the global existence for the 2-D Euler equation in critical Besov space $B^1_{\infty,1}$ \cite{Bahouri2011}.
\end{rema}
\begin{theo}[Large time behaviour]\label{th2}
Let $(u,\tau)$ be a strong solution of \eqref{eq0} with $\nu=0$ the initial data $(u_0,\tau_0)$ under the condition in Theorem \ref{th3}. In addition, if $(u_0,\tau_0)\in L^1$, then there exists $C>0$ such that for every $t>0$, there holds
\begin{align}\label{decay}
\|(u,\tau)\|_{H^1} \leq C(1+t)^{-\frac{1}{2}}.
\end{align}
\end{theo}
\begin{rema}
Notice that Theorem \ref{th3} don't provide any information for the global solution of \eqref{eq0} in $L^\infty([0,\infty);H^s)$ with some large initial data. However, by virtue of the Fourier splitting method and the time weighted energy estimate, we can prove the large time behaviour by taking full advantage of the $H^1$ energy estimation \eqref{estimate} and the low-frequency assumption $(u_0,\tau_0)\in L^1$. The proof does not involve the higher derivative, which is useful in studying large time behaviour of global solutions with some large initial data.
\end{rema}
\begin{rema}
The conclusions in Theorem \ref{th2} and Theorems \ref{th1}, \ref{th13} reveal the essential difference between \eqref{eq0} and \eqref{eq2}. More precisely, the solutions $(u,\tau)$ of \eqref{eq0} with $\nu=0$ decay in $H^1$, while the solutions $u$ of \eqref{eq2} are bounded in $H^1$. Moreover, $u$ conserve in $L^2$ whenever $\tau=0$. Such observation reflects the obstacle of global approximation in $H^s$ between \eqref{eq0} and \eqref{eq2}.
\end{rema}

\subsection{Motivation and main ideas}
Global well-posedness with $d=2$ and long time behavior for polymeric models were noticed by F. Lin \cite{Lin2020} and N. Masmoudi \cite{Masmoudi2013,2016Equations}. To our best knowledge, global well-posedness for the co-rotation inviscid Oldroyd-B model \eqref{eq2} and large time behaviour of global strong solutions with large data for the noncorotation inviscid Oldroyd-B model \eqref{eq0} have not been studied yet. To get global existence for \eqref{eq2} with $d=2$, the key point is to obtain the uniform estimate of $\|\Omega\|_{L^{\infty}}$. However, due to the linear term
$\nabla \times {\rm div}~\tau$ breaking the conservation laws, it is difficult to get global estimate of $\|\Omega\|_{L^{\infty}}$ from the following equation
\begin{align*}
	\frac{d}{dt}\Omega + u\cdot\nabla\Omega = \nabla \times {\rm div}~\tau.
\end{align*}

In this paper, we firstly study about global solutions for \eqref{eq2} with large data in $H^s$. The proof is based on the bootstrap argument in \cite{Wei2020}. To prove global existence, we derive the energy estimate and B-K-M criterion for \eqref{eq2} in $H^s$. The main difficult in the proof is to prove the boundedness of vorticity from \eqref{eq2}. Motivated by \cite{2015Elgindi}, we can cancel $\nabla \times{\rm div}~\tau$ and $\Delta\tau$ by virtue of the structural trick
$$\Gamma=\mu\Omega-R\tau,$$
where $R=-(\Delta)^{-1}{\rm curl}~{\rm div}$. Then, we deduce from \eqref{eq2} that
\begin{align*}
\frac{d}{dt}\Gamma + u\cdot\nabla\Gamma = aR\tau + RQ(\Omega,\tau) + [R,u\cdot\nabla]\tau.
\end{align*}
Different from \cite{2015Elgindi}, there is no damping phenomenon for $\Gamma$ or $\Omega$ for lack of $D(u)$. We thus fail to use the bootstrap argument as in \cite{2015Elgindi}. Fortunately, the disappearance of $D(u)$ leads to exponential dissipation for $\tau$ in $H^1$. The effect of exponential dissipation of $\tau$ is essential in the estimation of $\Gamma$. We finish the proof of global existence with large data in $H^s$ by deriving the $L^{\infty}$ estimate for $\Omega$.

To our best knowledge, there is still no result related to global existence of the inviscid Hooke models \eqref{eq1} with $\nu=0$. As a corollary of Theorem \ref{th1}, we prove the global existence of \eqref{eq1} with large data in $H^s$. It's worth mentioning that the estimate of $\langle q\rangle^{n}\nabla^m_qg$ in $L^{\infty}(\mathcal{L}^2)$ is the key to the proof of global existence.

Furthermore, we establish local existence for \eqref{eq2} in $B^1_{\infty,1}\times B^0_{\infty,1}$ and present the global existence with large data in
$$(H^1\cap B^1_{\infty,1}) \times (H^1\cap B^0_{\infty,1}).$$
The proof of global existence is based on the refined estimate in Besov space with index 0 and the $H^1\times (H^1\cap L^\infty)$ boundedness for $(u,\tau)$, where we use the fact that
$$\langle u\cdot\nabla u,\Delta u\rangle = 0,$$
for $d=2$ and ${\rm div}~u=0$.
By virtue of the refined estimate in Besov spaces with index $0$, the authors \cite{Bahouri2011} prove the global well-posedness for the Euler equation in the borderline case and obtain the exponential growth estimate of vorticity in $B^0_{\infty,1}$. Considering global existence for \eqref{eq2} in critical Besov space, the main difficult for us is to estimate external force in the equation of $\Gamma$. We find that exponential
dissipation for $\tau$ is crucial, which can prevent the exponential growth of external force in the equation of $\Gamma$. Thus we obtain the exponential growth estimate of $\Gamma$ in $B^0_{\infty,1}$, which implies the global existence for \eqref{eq2}.

Finally, we study about large time behaviour for \eqref{eq0} with $\nu=0$ and $d=2$. Since the structural trick $\Gamma$ transfer dissipation from $\tau$ to $u$ for \eqref{eq0}, we obtain the dissipation energy estimate for $(u,\tau)$ which is useful to prove large time behaviour with large data. By the Fourier splitting method, for any $l\in N^{+}$ and $t>0$, we get initial time decay rate
\begin{align*}
\|(u,\tau)\|_{H^1} \leq C\ln^{-l}(e+t),
\end{align*} see \cite{Luo-Yin,Schonbek1985}. By virtue of the time weighted energy estimate and the logarithmic decay rate, then we improve the time decay rate to $\|(u,\tau)\|_{H^1} \leq C(1+t)^{-\frac{1}{2}}.$

The paper is organized as follows. In Section 2 we introduce some notations and give some preliminaries which will be used in the sequel. In Section 3 we prove that the 2-D co-rotation inviscid Oldroyd-B model admits a unique global strong solution in Sobolev space. As a corollary, we prove global existence of the inviscid Hooke model near equilibrium. In Section 4 we prove that the 2-D co-rotation inviscid Oldroyd-B model admits a global strong solution in critical Besov space.
In Section 5 we study the $H^1$ decay rate of global strong solutions to the noncorotation inviscid Oldroyd-B model for $d=2$ by virtue of the Fourier splitting method.

\section{Preliminaries}
In this section we introduce some notations and useful lemmas which will be used in the sequel.

Let $\psi_{\infty}(q)=\frac{e^{-\mathcal{U}(q)}}{\int_{\mathbb{R}^{d}}e^{-\mathcal{U}(q)}dq}$.
For $p\geq1$, we denote by $\mathcal{L}^{p}$ the space
$$\mathcal{L}^{p}=\big\{\psi \big|\|\psi\|^{p}_{\mathcal{L}^{p}}=\int_{\mathbb{R}^{d}} \psi_{\infty}|\psi|^{p}dq<\infty\big\}.$$

We will use the notation $L^{p}_{x}(\mathcal{L}^{p'})$ to denote $L^{p}(\mathbb{R}^{d};\mathcal{L}^{p'}):$
$$L^{p}_{x}(\mathcal{L}^{p'})=\big\{\psi \big|\|\psi\|_{L^{p}_{x}(\mathcal{L}^{p'})}=(\int_{\mathbb{R}^{d}}(\int_{\mathbb{R}^{d}} \psi_{\infty}|\psi|^{p'}dq)^{\frac{p}{p'}}dx)^{\frac{1}{p}}<\infty\big\}.$$

We now recall the Littlewood-Paley decomposition theory and Besov spaces.
\begin{prop}\cite{Bahouri2011}\label{prop0}
Let $\mathcal{C}$ be the annulus $\{\xi\in\mathbb{R}^d:\frac 3 4\leq|\xi|\leq\frac 8 3\}$. There exist radial functions $\chi$ and $\varphi$, valued in the interval $[0,1]$, belonging respectively to $\mathscr{D}(B(0,\frac 4 3))$ and $\mathscr{D}(\mathcal{C})$, and such that
$$ \forall\xi\in\mathbb{R}^d,\ \chi(\xi)+\sum_{j\geq 0}\varphi(2^{-j}\xi)=1, $$
$$ \forall\xi\in\mathbb{R}^d\backslash\{0\},\ \sum_{j\in\mathbb{Z}}\varphi(2^{-j}\xi)=1, $$
$$ |j-j'|\geq 2\Rightarrow\mathrm{Supp}\ \varphi(2^{-j}\cdot)\cap \mathrm{Supp}\ \varphi(2^{-j'}\cdot)=\emptyset, $$
$$ j\geq 1\Rightarrow\mathrm{Supp}\ \chi(\cdot)\cap \mathrm{Supp}\ \varphi(2^{-j}\cdot)=\emptyset. $$
The set $\widetilde{\mathcal{C}}=B(0,\frac 2 3)+\mathcal{C}$ is an annulus, and we have
$$ |j-j'|\geq 5\Rightarrow 2^{j}\mathcal{C}\cap 2^{j'}\widetilde{\mathcal{C}}=\emptyset. $$
Further, we have
$$ \forall\xi\in\mathbb{R}^d,\ \frac 1 2\leq\chi^2(\xi)+\sum_{j\geq 0}\varphi^2(2^{-j}\xi)\leq 1, $$
$$ \forall\xi\in\mathbb{R}^d\backslash\{0\},\ \frac 1 2\leq\sum_{j\in\mathbb{Z}}\varphi^2(2^{-j}\xi)\leq 1. $$
\end{prop}

$\mathscr{F}$ represents the Fourier transform and  its inverse is denoted by $\mathscr{F}^{-1}$.
Let $u$ be a tempered distribution in $\mathcal{S}'(\mathbb{R}^d)$. For all $j\in\mathbb{Z}$, define
$$
\Delta_j u=0\,\ \text{if}\,\ j\leq -2,\quad
\Delta_{-1} u=\mathscr{F}^{-1}(\chi\mathscr{F}u),\quad
\Delta_j u=\mathscr{F}^{-1}(\varphi(2^{-j}\cdot)\mathscr{F}u)\,\ \text{if}\,\ j\geq 0,\quad
S_j u=\sum_{j'<j}\Delta_{j'}u.
$$
Then the Littlewood-Paley decomposition is given as follows:
$$ u=\sum_{j\in\mathbb{Z}}\Delta_j u \quad \text{in}\ \mathcal{S}'(\mathbb{R}^d). $$
Let $s\in\mathbb{R},\ 1\leq p,r\leq\infty.$ The nonhomogeneous Besov spaces $B^s_{p,r}$ and $B^s_{p,r}(\mathcal{L}^{p'})$ are defined by
$$ B^s_{p,r}=\{u\in S':\|u\|_{B^s_{p,r}}=\Big\|(2^{js}\|\Delta_j u\|_{L^p})_j \Big\|_{l^r(\mathbb{Z})}<\infty\}, $$
and
$$ B^s_{p,r}(\mathcal{L}^{p'})=\{\phi\in S':\|\phi\|_{B^s_{p,r}(\mathcal{L}^{p'})}=\Big\|(2^{js}\|\Delta_j \phi\|_{L_{x}^{p}(\mathcal{L}^{p'})})_j \Big\|_{l^r(\mathbb{Z})}<\infty\}.$$
For any positive time $T$, the Time-Space Besov Spaces are defined by
\begin{align*}
L^{\rho}_T(B^s_{p,r}) = \{u\in S':\|u\|_{L^{\rho}_T(B^s_{p,r})} = \left\|\left\|2^{js}\|\Delta_ju\|_{L^p}\right\|_{l^r({\mathbb{Z}})}\right\|_{L^{\rho}_T}<\infty\},
\end{align*}
and
\begin{align*}
\tilde{L}^{\rho}_T(B^s_{p,r}) = \{u\in S':\|u\|_{\tilde{L}^{\rho}_T(B^s_{p,r})} = \left\|2^{js}\|\Delta_ju\|_{L^{\rho}_T(L^p)}\right\|_{l^r({\mathbb{Z}})}<\infty\}.
\end{align*}
Moreover, the following embedding relationships hold:
$$
L^{\rho}_T(B^s_{p,r})\hookrightarrow \tilde{L}^{\rho}_T(B^s_{p,r}) ~~~~\text{if}~~~~r\geq\rho~~~~\text{and}~~~~
\tilde{L}^{\rho}_T(B^s_{p,r})\hookrightarrow L^{\rho}_T(B^s_{p,r})  ~~~~\text{if}~~~~r\leq\rho.$$

We agree that $\nabla$ stands for $\nabla_x$ and ${\rm div}$ stands for ${\rm div}_x$.
Now we state some useful estimates in the study of transport-diffusion equations, which are
useful to the proofs of our main theorem later.
\begin{equation}\label{transport}
\left\{\begin{array}{l}
    f_t+v\cdot\nabla f-\nu\Delta f=g,\ x\in\mathbb{R}^d,\ t>0,\\
    f(0,x)=f_0(x).
\end{array}\right.
\end{equation}
\begin{lemm}\label{transportdiffusion}\cite{Bahouri2011}
Let $1\leq p_1\leq p\leq\infty, \ 1\leq r\leq\infty, \ s\geq-1-d\min(\frac{1}{p_1},\frac{1}{p'}).$ Let ${\rm div}~v=0$.
There exists a constant $C$ and $1\leq \rho_1\leq \rho \leq \infty$ such that for all solutions $f\in L^{\infty}([0,T];B^s_{p,r})$ of \eqref{transport} with initial data $f_0$ in $B^s_{p,r}$, and $g$ in $\tilde{L}^{\rho_1}([0,T);B^{s-2+\frac{2}{\rho_1}}_{p,r})$, we have
\begin{align}
\nu^{\frac{1}{\rho}}\|f(t)\|_{\tilde{L}^{\rho}_T(B^{s+\frac{2}{
			\rho}}_{p,r})}&\leq Ce^{C(1+\nu T)^{\frac{1}{
\rho}}V_{p_1}(T)}\Big((1+\nu T)^{\frac{1}{
\rho}}\|f_0\|_{B^s_{p,r}}\\ \notag
&~~~+(1+\nu T)^{1+\frac{1}{\rho}-\frac{1}{\rho_1}}\nu^{\frac{1}{\rho_1}-1} \|g\|_{\tilde{L}^{\rho_1}_T(B^{s-2+\frac{2}{\rho_1}}_{p,r})}\Big),
\end{align}
with
\begin{equation*}
  V'_{p_1}(t)=\left\{\begin{array}{ll}
  \|\nabla v\|_{B^{\frac{d}{p_1}}_{p_1,r}},\ &\text{if}\ s<1+\frac{d}{p_1}, \\
  \|\nabla v\|_{B^{s-1}_{p_1,r}},\ &\text{if}\ s>1+\frac d {p_1}\ \text{or}\ (s=1+\frac d {p_1}, \ r=1).   \\
  \end{array}\right.
\end{equation*}
\end{lemm}
The following refined estimate in Besov spaces with index $0$ is crucial to estimate $\Gamma$.
\begin{lemm}\cite{Bahouri2011}
Assume that $v$ is divergence-free and that $f$ satisfies \eqref{transport} with $\nu=0$. There exists a constant $C$, depending
only on d, such that for all $1 \leq p, r \leq \infty$ and $t\in[0, T]$, we have
\begin{align}\label{Lemma1''}
\|f\|_{\tilde{L}^{\infty}_t (B^0_{p,r})} \leq C	(\|f_0\|_{B^0_{
p,r}} + \|g\|_{L^1_t(B^0_{p,r})})(1 + V (t)),
\end{align}
with $V(t) =  C\int_0^t \|\nabla v\|_{L^{\infty}}ds$.
\end{lemm}

We have the following product laws.
\begin{lemm}\label{Lemma1}\cite{Bahouri2011}
For any $\epsilon>0$, there exists $C>0$  such that
$$ \|uv\|_{B^0_{\infty,1}}\leq C(\|u\|_{L^{\infty}}\|v\|_{B^0_{\infty,1}}+\|u\|_{B^0_{\infty,1}}\|v\|_{B^\epsilon_{\infty,\infty}}). $$
\end{lemm}
We introduce the following lemma to describe the action of the heat equation.
\begin{lemm}\label{Lemma1'}\cite{Bahouri2011}
Let $\mathcal{C}$ be an annulus. Positive constants $c$ and $C$ exist such that for any $p\in[1,+\infty]$ and any couple $(t,\lambda)$ of positive real numbers, we have
\begin{align}
Supp~\hat{u} \subset \lambda\mathcal{C} \Rightarrow \|e^{t\Delta}u\|_{L^p} \leq Ce^{-ct\lambda^2}\|u\|_{L^p}.
\end{align}
\end{lemm}
The following commutator lemma is useful to estimate $\Gamma$.
\begin{lemm}\cite{2015Elgindi,2011Hmidi}\label{Lemma2}
Let ${\rm div}~u=0$ and $R=\Delta^{-1}{\rm curl}~{\rm div}$.\\
(1) For every $(p,r)\in[2,\infty)\times[1,\infty]$, there exists a constant $C=C(p,r)$ such that
\begin{align}
\|[R,u\cdot\nabla]\tau\|_{B^{0}_{p,r}}\leq C\|\nabla u\|_{L^p}(\|\tau\|_{B^0_{\infty,r}}+\|\tau\|_{L^p}).
\end{align}
(2) For every $(r,p)\in[1,\infty]\times(1,\infty)$ and $\varepsilon>0$, there exists a constant $C=C(r,p,\varepsilon)$ such that
\begin{align}
\|[R,u\cdot\nabla]\tau\|_{B^{0}_{\infty,r}}\leq C(\|\Omega\|_{L^{\infty}}+\|\Omega\|_{L^p})(\|\tau\|_{B^\varepsilon_{\infty,r}}+\|\tau\|_{L^p}).
\end{align}
\end{lemm}

Let $\Lambda^s f=\mathcal{F}^{-1}(|\xi|^s \widehat{f})$.
The following lemma is the Gagliardo-Nirenberg inequality of Sobolev type.
\begin{lemm}\cite{1959On}\label{Lemma3}
Let $d\geq2,~p\in[2,+\infty)$ and $0\leq s,s_1\leq s_2$, then there exists a constant $C$ such that
$$\|\Lambda^{s}f\|_{L^{p}}\leq C \|\Lambda^{s_1}f\|^{1-\theta}_{L^{2}}\|\Lambda^{s_2} f\|^{\theta}_{L^{2}},$$
where $0\leq\theta\leq1$ and $\theta$ satisfy
$$ s+d(\frac 1 2 -\frac 1 p)=s_1 (1-\theta)+\theta s_2.$$
Note that we require that $0<\theta<1$, $0\leq s_1\leq s$, when $p=\infty$.
\end{lemm}

We introduce a commutator lemma.
\begin{lemm}\cite{Moser1966A}\label{Lemma4}
Let $s\geq 1$, $p,p_1,p_4\in (1,\infty)$ and $\frac 1 p =\frac 1 {p_1}+\frac 1 {p_2}=\frac 1 {p_3}+\frac 1 {p_4}$, then we have
\begin{align}
\|[\Lambda^s, f]g\|_{L^p}\leq C(\|\Lambda^{s}f\|_{L^{p_1}}\|g\|_{L^{p_2}}+\|\nabla f\|_{L^{p_3}}\|\Lambda^{s-1}g\|_{L^{p_4}}),
\end{align}
and
\begin{align}
\|[\Lambda^s, f]g\|_{L^p(\mathcal{L}^2)}\leq C(\|\Lambda^{s}f\|_{L^{p_1}}\|g\|_{L^{p_2}(\mathcal{L}^2)}+\|\nabla f\|_{L^{p_3}}\|\Lambda^{s-1}g\|_{L^{p_4}(\mathcal{L}^2)}).
\end{align}
\end{lemm}
The following lemma will be useful in the proof of global existence for the inviscid Hooke model.
\begin{lemm}\cite{2017Global}\label{Lemma7}
Let $\langle q \rangle=\sqrt{1+q^2}$. Assume $g\in H^s(\mathcal{L}^{2})$ with $\int_{\mathbb{R}^{d}} g\psi_\infty dq=0$ and $s\geq 0$, then there exists a constant $C$ such that
\begin{align}
\|\nabla_q\mathcal{U}g\|_{H^s(\mathcal{L}^{2})}+\|qg\|_{H^s(\mathcal{L}^{2})}\leq C\|\nabla_qg\|_{H^s(\mathcal{L}^{2})},
\end{align}
and
\begin{align}
\|q\nabla_q\mathcal{U}g\|_{H^s(\mathcal{L}^{2})}+\||q|^2g\|_{_{H^s(\mathcal{L}^{2})}}\leq C\|\langle q\rangle\nabla_qg\|_{H^s(\mathcal{L}^{2})}.
\end{align}
\end{lemm}
The following lemmas are about Calderon-Zygmund operator.
\begin{lemm}\cite{Bahouri2011,2015Elgindi}\label{Lemma5}
(1) For any $a\in[1,\infty)$ and $b\in[1,\infty]$, there exists positive constant $C$ such that
\begin{align}
\|\Delta_{-1}\nabla v\|_{L^{\infty}} \leq C\min\{\|\Omega\|_{L^a},\|v\|_{L^b}\}.
\end{align}
(2) For all $s\in\mathbb{R}$ and $1\leq p,r\leq\infty$, there exists a constant $C$ such that
\begin{align}
\|(Id-\Delta_{-1}) Rf\|_{B^s_{p,r}} \leq C\|f\|_{B^s_{p,r}}.
\end{align}
\end{lemm}
\begin{lemm}\cite{Bahouri2011,2015Elgindi}\label{Lemma5'}
There exists positive constant $C$ such that
\begin{align}
\|u\|_{L^{\infty}} \leq C\big(\|u\|_{L^2} + \|\Omega\|_{L^{\infty}}\big).
\end{align}
\end{lemm}
We now introduce a useful interpolation inequality.
\begin{lemm}\label{Lemma6}
Let $s>\frac d 2$. Then there exist $C>0$ such that
\begin{align*}
\|u\|_{L^\infty}\leq \|u\|_{B^0_{\infty,1}}\leq C\|u\|_{B^0_{\infty,\infty}}\ln(e+\|u\|_{H^s}) + C.
\end{align*}
\end{lemm}
\begin{proof}
According to the Littlewood-Paley decomposition theory, we have
\begin{align*}
\|u\|_{L^\infty}\leq\|u\|_{B^0_{\infty,1}}=\mathop{\sum}\limits_{-1 \leq j \leq N} \|\Delta_ju\|_{L^\infty} + \mathop{\sum}\limits_{j \geq N} \|\Delta_ju\|_{L^\infty},
\end{align*}
for integer $N>0$ which will be chosen later on. There exist $C>0$ such that
\begin{align*}
\mathop{\sum}\limits_{-1 \leq j \leq N}\|\Delta_ju\|_{L^\infty} \leq CN \|u\|_{B^0_{\infty,\infty}},
\end{align*}
and
\begin{align*}
\mathop{\sum}\limits_{j \geq N}\|\Delta_ju\|_{L^\infty} \leq C2^{-N(s-\frac d 2)}\|u\|_{H^s}.
\end{align*}
Consider $N = [\frac {\ln(e+\|u\|_{H^s})}{s-\frac d 2}]+1$, then we complete the proof of Lemma \ref{Lemma6}.
\end{proof}

\section{Global solutions for the co-rotation case in Sobolev space}
In this section, we are concerned with global solutions to the 2-D co-rotation inviscid Oldroyd-B in Sobolev space. We divide it into three subsections to prove Theorem \ref{th1}.

\subsection{Energy estimates}
From now on, we derive the energy estimate which is useful to prove global existence for \eqref{eq2}. We prove conservation laws and boundness for \eqref{eq2} in the following propositions.
\begin{prop}\label{7lemm1'}
Set $p\in[2,\infty]$. Suppose $(u,\tau)$ is a smooth solution to \eqref{eq2} with $\tau_0$ in $L^p$. Then we obtain
\begin{align}
\|\tau\|_{L^p} \leq \|\tau_0\|_{L^p}e^{-at}.
\end{align}
\end{prop}
\begin{proof}
Let $\tilde{\tau}^{ij}=\tau^{ij}e^{at}$, we infer from $(\ref{eq2})_2$ that
\begin{align}\label{7ineq1'}
\partial_t\tilde{\tau}^{ij} + u\cdot\nabla\tilde{\tau}^{ij} + Q(\Omega^{ik},\tilde{\tau}^{kj}) = \Delta\tilde{\tau}^{ij}.
\end{align}
Applying inner product with $\tilde{\tau}^{ij}|\tilde{\tau}|^{p-2}$ to \eqref{7ineq1'} and summing up $i,j$, we get
\begin{align}\label{7ineq2'}
\frac{1}{p}\frac{d}{dt}\|\tilde{\tau}\|^p_{L^p}  = \sum^2_{i,j=1}\int_{\mathbb{R}^2}\tilde{\tau}^{ij}|\tilde{\tau}|^{p-2}\Delta\tilde{\tau}^{ij}dx,
\end{align}
where we use the fact that $$\sum^2_{i,j=1}\int_{\mathbb{R}^2}\tilde{\tau}^{ij}|\tilde{\tau}|^{p-2}Q(\Omega^{ik},\tilde{\tau}^{kj})dx=0.$$
Notice that
\begin{align}\label{7ineq3'}
\int_{\mathbb{R}^2}\tilde{\tau}^{ij}|\tilde{\tau}|^{p-2}\Delta\tilde{\tau}^{ij}dx &= - \int_{\mathbb{R}^2}\nabla^k\tilde{\tau}^{ij}|\tilde{\tau}|^{p-2}\nabla^k\tilde{\tau}^{ij}dx - (p-2)\int_{\mathbb{R}^2}\tilde{\tau}^{ij}\tilde{\tau}^{ij}\nabla^k\tilde{\tau}^{ij}|\tilde{\tau}|^{p-4}\nabla^k\tilde{\tau}^{ij}dx\\ \notag
&= - \int_{\mathbb{R}^2}(\nabla^k\tilde{\tau}^{ij})^2|\tilde{\tau}|^{p-2}dx - \frac{p-2}{4}\int_{\mathbb{R}^2}(\nabla^k(\tilde{\tau}^{ij})^2)^2|\tilde{\tau}|^{p-4}dx\leq 0.
\end{align}
According to \eqref{7ineq2'} and \eqref{7ineq3'} with $2\leq p<\infty$, we obtain
\begin{align}\label{7ineq4'}
\frac{1}{p}\frac{d}{dt}\|\tilde{\tau}\|^p_{L^p} \leq 0,
\end{align}
which implies that
\begin{align}\label{7ineq6'}
\|\tilde{\tau}\|_{L^p} \leq \|\tau_0\|_{L^p}.
\end{align}
Taking $p\rightarrow\infty$, one can deduce that \eqref{7ineq6'} is valid for $p=\infty$. Therefore, we complete the proof of Proposition \ref{7lemm1'}.
\end{proof}
\begin{prop}\label{prop1}
Let $(u,\tau)\in C([0,T];H^1)\times C([0,T];H^1)\cap L^2([0,T];H^{2})$ be a smooth solution for \eqref{eq2}. Then we obtain
\begin{align}\label{ineq1}
\|u\|_{L^2} \leq  \|u_0\|_{L^2} + (4\mu a)^{-\frac{1}{2}}\|\tau_0\|_{L^2},~~~~e^{2at}\|\tau\|^2_{L^2} + 2\mu\int_0^t e^{2as}\|\nabla\tau\|^2_{L^2} ds= \|\tau_0\|^2_{L^2}.
\end{align}
Moreover, for any $t\in[0,T]$, if $\|\nabla u(t)\|_{L^2} \leq 4c\kappa$ with $\kappa=\min\{a,\mu\}$ and sufficiently small constant c,
then we obtain
\begin{align}\label{ineq2}
e^{at}\|\tau\|^2_{H^1} + \mu\int_0^t e^{as} \|\nabla\tau\|^2_{H^1} ds \leq \|\tau_0\|^2_{H^1},
\end{align}
and
\begin{align}\label{ineq3}
\|\nabla u\|_{L^2} \leq \|\nabla u_0\|_{L^2}+(\mu a)^{-\frac{1}{2}} \|\tau_0\|_{H^1}.
\end{align}
\end{prop}
\begin{proof}
Firstly, we consider the $L^2$ estimate of $(u,\tau)$. Taking the $L^2$ inner product with $\tau$ to $(\ref{eq2})_2$, we obtain
\begin{align}\label{eq3}
\frac{1}{2}\frac{d}{dt}\|\tau\|^2_{L^2}+a\|\tau\|^2_{L^2}+\mu\|\nabla\tau\|^2_{L^2}=0,
\end{align}
which implies that
\begin{align}\label{eq4}
e^{2at}\|\tau\|^2_{L^2}+2\mu\int_0^t e^{2as}\|\nabla\tau\|^2_{L^2}ds=\|\tau_0\|^2_{L^2}.
\end{align}
Taking the $L^2$ inner product with $u$ to $(\ref{eq2})_1$ and using ${\rm div}~u=0$, we obtain
\begin{align}\label{ineq4}
\frac{d}{dt}\|u\|_{L^2} \leq \|\nabla\tau\|_{L^2}.
\end{align}
Integrating \eqref{ineq4} over $[0,t]$ with $s$, we deduce that
\begin{align*}
\|u\|_{L^2} &\leq \|u_0\|_{L^2}  + \int_0^t \|\nabla\tau\|_{L^2} ds \\ \notag
&\leq \|u_0\|_{L^2} + (\int_0^t e^{2as}\|\nabla\tau\|^2_{L^2} ds)^{\frac{1}{2}}(\int_0^t e^{-2as}ds)^{\frac{1}{2}} \\ \notag
&\leq  \|u_0\|_{L^2} + (4\mu a)^{-\frac{1}{2}}\|\tau_0\|_{L^2}.
\end{align*}
Notice that $(u,\tau)$ are bound in $L^2$ for any initial value. 

Taking the $L^2$ inner product with $-\Delta\tau$ to $(\ref{eq2})_2$ and using Lemma \ref{Lemma3}, we have
\begin{align}\label{ineq5}
\frac{1}{2}\frac{d}{dt}\|\nabla\tau\|^2_{L^2} + a\|\nabla\tau\|^2_{L^2} + \mu\|\nabla^2\tau\|^2_{L^2}&=\langle u\cdot\nabla\tau,\Delta\tau\rangle+\langle Q(\Omega,\tau),\Delta\tau\rangle \\ \notag
&\leq C\|\nabla u\|_{L^2}\|\nabla\tau\|_{L^2}\|\nabla^2\tau\|_{L^2}\\ \notag
&~~~+C\|\Omega\|_{L^2}(\|\tau\|_{L^2} + \|\nabla^2\tau\|_{L^2})\|\Delta\tau\|_{L^2}.
\end{align}
Adding up \eqref{eq3} and \eqref{ineq5}, we infer that
\begin{align}
\frac{1}{2}\frac{d}{dt}\|\tau\|^2_{H^1}+a\|\tau\|^2_{H^1}+\mu\|\nabla\tau\|^2_{H^1}\leq C(\|\nabla u\|_{L^2}\|\nabla\tau\|_{L^2}+\|\Omega\|_{L^2}\|\tau\|_{H^2})\|\nabla^2\tau\|_{L^2}.
\end{align}
Assume that $\|\nabla u\|_{L^2} \leq 4c\min\{a,\mu\}$ with sufficiently small constant c, then we obtain
\begin{align}\label{ineq6}
\frac{d}{dt}\|\tau\|^2_{H^1} + a\|\tau\|^2_{H^1} + \mu\|\nabla\tau\|^2_{H^1} \leq 0,
\end{align}
which implies that
\begin{align}\label{ineq7}
e^{at}\|\tau\|^2_{H^1} + \mu\int_0^t e^{as} \|\nabla\tau\|^2_{H^1} ds \leq \|\tau_0\|^2_{H^1}.
\end{align}
We now consider the $L^2$ estimate of $\nabla u$. Taking the $L^2$ inner product with $-\Delta u$ to $(\ref{eq2})_1$, we can deduce that $\langle u\cdot\nabla u,-\Delta u\rangle = 0$ with $d=2$ and ${\rm div}~u=0$. Then we have
\begin{align}\label{ineq8}
\frac{d}{dt}\|\nabla u\|_{L^2} \leq \|\nabla^2 \tau\|_{L^2}.
\end{align}
Integrating \eqref{ineq8} over $[0,t]$ with $s$ and using \eqref{ineq7}, we deduce that
\begin{align*}
\|\nabla u\|_{L^2} &\leq \|\nabla u_0\|_{L^2} + \int_0^t \|\nabla^2 \tau\|_{L^2} ds \\
&\leq \|\nabla u_0\|_{L^2} + (\mu a)^{-\frac{1}{2}} \|\tau_0\|_{H^1}.
\end{align*}
This together with \eqref{eq4} and \eqref{ineq7} completes the proof of Proposition \ref{prop1}.
\end{proof}

\begin{coro}\label{coro1}
Under the conditions in Proposition \ref{prop1}, we deduce the following estimates:
\begin{align}
\left\{\begin{array}{l}
\int_0^t \|\tau\|_{H^1}ds \leq 2a^{-1}\|\tau_0\|_{H^1},\\
\int_0^t \|\tau\|_{H^2}ds \leq (a^{-1}+(\mu a)^{-\frac{1}{2}})\|\tau_0\|_{H^1},\\
\int_0^t \|\tau\|^2_{H^2}ds \leq (a^{-1}+\mu^{-1})\|\tau_0\|^2_{H^1}.
\end{array}\right.
\end{align}
\end{coro}
\begin{proof}
Using \eqref{ineq7}, we can deduce that
\begin{align*}
\int_0^t \|\tau\|_{H^1}ds \leq \|\tau_0\|_{H^1}\int_0^t e^{-\frac a 2 s}ds \leq 2a^{-1}\|\tau_0\|_{H^1},
\end{align*}
and
\begin{align*}
\int_0^t \|\tau\|^2_{H^2}ds &\leq \int_0^t \|\tau\|^2_{H^1}ds+\int_0^t \|\nabla^2\tau\|^2_{L^2}ds \\
&\leq (a^{-1}+\mu^{-1})\|\tau_0\|^2_{H^1}.
\end{align*}
Using \eqref{eq4} and \eqref{ineq7}, we have
\begin{align*}
\int_0^t \|\tau\|_{H^2}ds &\leq \int_0^t \|\tau\|_{L^2}ds+(\int_0^t e^{-as}ds)^{\frac{1}{2}}(\int_0^t e^{as}\|\nabla\tau\|^2_{H^1}ds)^{\frac{1}{2}} \\
&\leq  (a^{-1}+(\mu a)^{-\frac{1}{2}})\|\tau_0\|_{H^1}.
\end{align*}
We thus complete the proof of Corollary \ref{coro1}.
\end{proof}

\subsection{B-K-M criterion}
In Proposition \ref{prop1}, it's clear that $u$ is merely bound in $L^2$ while $\tau$ decays exponentially in $L^2$. Then we can state a blow-up criterion for \eqref{eq2} which depends on $\|\Omega\|_{L^\infty}$ in the following proposition.
\begin{prop}\label{prop2}
Assume that $d=2$, $s>2$, $a>0$ and $\mu>0$. Let $(u,\tau)$ be a strong solution of \eqref{eq2} with the initial data $(u_0,\tau_0)\in H^s$. Suppose that $T^\ast$ is the maximal existence time, then the solution blows up in finite time $T^\ast<\infty$ if and only if
\begin{align}\label{ineq9}
\int_{0}^{T^\ast}\|\Omega(t)\|^2_{L^\infty}dt=\infty.
\end{align}
\end{prop}
\begin{proof}
Applying $\Lambda^s$ to $(\ref{eq2})_1$, taking the $L^2$ inner product with $\Lambda^s u$ and using Lemma \ref{Lemma4}, we have
\begin{align}\label{ineq9'}
\frac{1}{2}\frac{d}{dt}\|\Lambda^s u\|^2_{L^2} &= -\langle[\Lambda^s, u\cdot\nabla] u,\Lambda^s u\rangle + \langle {\rm div}~\Lambda^s\tau,\Lambda^s u\rangle\\ \notag
 &\leq C\|\nabla u\|_{L^{\infty}}\|\Lambda^s u\|^2_{L^2} + C_{\mu}\|\Lambda^s u\|^2_{L^2} + \frac{\mu}{4}\|\nabla\Lambda^{s} \tau\|^2_{L^2},
\end{align}
where we take $C_\mu=\frac C \mu$.
Applying $\Lambda^s$ to $(\ref{eq2})_2$, taking the $L^2$ inner product with $\Lambda^s \tau$ and using Lemmas \ref{Lemma3}-\ref{Lemma4}, we obtain
\begin{align}\label{ineq10}
\frac{1}{2}\frac{d}{dt}\|\Lambda^s \tau\|^2_{L^2} + a\|\Lambda^s \tau\|^2_{L^2} + \mu\|\nabla\Lambda^{s} \tau\|^2_{L^2}&=-\langle[\Lambda^s,u]\tau,\nabla\Lambda^{s} \tau\rangle- \langle\Lambda^sQ(\Omega,\tau),\Lambda^s \tau\rangle \\ \notag
&\leq C_{\mu}\|\tau\|^2_{L^{\infty}}\|\Lambda^s u\|^2_{L^2}+C_{\mu}\|u\|^2_{L^{\infty}}\|\Lambda^s\tau\|^2_{L^2}\\ \notag
&~~~+C_{\mu}\|\Lambda^{s-1}Q(\Omega,\tau)\|^2_{L^2} + \frac{\mu}{4}\|\nabla\Lambda^{s} \tau\|^2_{L^2}\\ \notag
&\leq C_{\mu}\|\tau\|^2_{L^{\infty}}\|\Lambda^s u\|^2_{L^2}+ \frac{\mu}{4}\|\nabla\Lambda^{s} \tau\|^2_{L^2} \\ \notag
&~~~+C_{\mu}(\|u\|^2_{L^{\infty}}
+\|\Omega\|^2_{L^\infty})\|\tau\|^2_{H^s}.
\end{align}
We infer from \eqref{eq3}, \eqref{ineq4}, \eqref{ineq9'} and \eqref{ineq10} that
\begin{align}\label{ineq11}
\frac{1}{2}\frac{d}{dt}\|(u,\tau)\|^2_{H^s} \leq (C\|\nabla u\|_{L^{\infty}} + C_{\mu}\|\tau\|^2_{L^{\infty}} + C_{\mu})\|u\|^2_{H^s} + C_{\mu}(\|u\|^2_{L^{\infty}} + \|\Omega\|^2_{L^{\infty}} )\|\tau\|^2_{H^s},
\end{align}
which implies that
\begin{align}\label{ineq12}
\|(u,\tau)\|^2_{H^s} &\leq \|(u_0,\tau_0)\|^2_{H^s} + \int_0^t (C\|\nabla u\|_{L^{\infty}} + C_{\mu}\|\tau\|^2_{L^{\infty}} + C_{\mu})\|u\|^2_{H^s}ds \\ \notag
&~~~+ \int_0^t C_{\mu}(\|u\|^2_{L^{\infty}} + \|\Omega\|^2_{L^{\infty}} )\|\tau\|^2_{H^s} ds.
\end{align}
Applying Gronwall's inequality, we deduce that
\begin{align}\label{ineq13}
e+\|(u,\tau)\|^2_{H^s} \leq (e + \|(u_0,\tau_0)\|^2_{H^s})e^{\int_0^t C\|\nabla u\|_{L^{\infty}} + C_{\mu}(\|\tau\|^2_{L^{\infty}}+\|u\|^2_{L^{\infty}} +\|\Omega\|^2_{L^{\infty}} + 1) ds}.
\end{align}
According to Lemma \ref{Lemma6}, we have
\begin{align}\label{ineq14}
\|\nabla u\|_{L^{\infty}}\leq C\|\nabla u\|_{B^{0}_{\infty,\infty}}\ln(e + \|u\|^2_{H^s})+C.
\end{align}
By virtue of \eqref{ineq13} and \eqref{ineq14}, we deduce that
\begin{align}\label{ineq15}
\ln( e + \|(u,\tau)\|^2_{H^s} ) &\leq \ln(e + \|(u_0,\tau_0)\|^2_{H^s}) + \int_0^t C_{\mu}(\|\tau\|^2_{L^{\infty}}+\|u\|^2_{L^{\infty}} +\|\Omega\|^2_{L^{\infty}} + 1) ds \\ \notag
&~~~+Ct+C\int_0^t \|\nabla u\|_{B^{0}_{\infty,\infty}}\ln(e + \|(u,\tau)\|^2_{H^s} )ds.
\end{align}
Applying Gronwall's inequality to \eqref{ineq15}, we infer that
\begin{align}\label{ineq16}
\ln( e+ \|(u,\tau)\|^2_{H^s} )&\leq(\ln(e + \|(u_0,\tau_0)\|^2_{H^s})+Ct)e^{C\int_0^t \|\nabla u\|_{B^{0}_{\infty,\infty}}ds} \\ \notag
&~~~+e^{C\int_0^t \|\nabla u\|_{B^{0}_{\infty,\infty}}ds}\int_0^t C_{\mu}(\|\tau\|^2_{L^{\infty}}+\|u\|^2_{L^{\infty}} +\|\Omega\|^2_{L^{\infty}} + 1) ds.
\end{align}
Assume that $T^\ast<\infty$ and $\int_{0}^{T^\ast}\|\Omega(t)\|^2_{L^\infty}dt<\infty$. By virtue of Lemmas \ref{Lemma5}, \ref{Lemma5'}, we obtain $$\|\nabla u\|_{B^{0}_{\infty,\infty}}+\|u\|_{L^{\infty}}\leq C(\|u\|_{L^2}+\|\Omega(t)\|_{L^\infty}).$$
According to \eqref{ineq16} and Propositions \ref{7lemm1'}, \ref{prop1}, then we have $(u,\tau)\in L^\infty([0,T^\ast);H^s)$, which contradicts the assumption that $T^\ast$ is the maximal existence time.
\end{proof}
\begin{rema}
We can deduce from Lemma \ref{Lemma4} that
\begin{align*}
C_{\mu}\|[\Lambda^s,u]\tau\|^2_{L^2}+C_{\mu}\|\Lambda^{s-1}Q(\Omega,\tau)\|^2_{L^2}&\leq C_{\mu}\|\tau\|^2_{L^{\infty}}\|\Lambda^s u\|^2_{L^2}+C_{\mu}\|\nabla u\|^2_{L^{4}}\|\Lambda^{s-1}\tau\|^2_{L^4}  \\
&\leq C_{\mu}\|\tau\|^2_{L^{\infty}}\|\Lambda^s u\|^2_{L^2}+C_{\mu}\|\Omega\|^2_{L^{4}}\|\tau\|^2_{H^s} \\
&\leq C_{\mu}\|\tau\|^2_{L^{\infty}}\|\Lambda^s u\|^2_{L^2}+C_{\mu}\|\nabla u\|_{L^2} \|\Omega\|_{L^{\infty}}\|\tau\|^2_{H^s}.
\end{align*}
One can see that \eqref{ineq16} can be rewritten as
\begin{align}\label{ineq16'}
\ln( e+ \|(u,\tau)\|^2_{H^s} )&\leq(\ln(e+ \|(u_0,\tau_0)\|^2_{H^s})+Ct)e^{C\int_0^t \|\nabla u\|_{B^{0}_{\infty,\infty}}ds} \\ \notag
&~~~+e^{C\int_0^t \|\nabla u\|_{B^{0}_{\infty,\infty}}ds}\int_0^t C_{\mu}(\|\tau\|^2_{L^{\infty}}+\|\nabla u\|_{L^2} \|\Omega\|_{L^{\infty}} + 1) ds,
\end{align}
which is of significance in the proof of Theorem \ref{th1}.
\end{rema}
\subsection{Global strong solutions}
\subsubsection{The inviscid Oldroyd-B model}
{\bf The proof of Theorem \ref{th1} :}  \\
The proof of the local well-posedness of \eqref{eq2} is standard. We thus omit it and present the result here. For any $T<T^\ast$, we have
$$u\in C([0,T];H^s),~~~~\tau\in C([0,T];H^s)\cap L^2([0,T];H^{s+1}).$$

To get global existence, the key point is to obtain the uniform estimate of $\|\Omega\|_{L^{\infty}}$. However, due to the linear term
$\nabla \times {\rm div}~\tau$, it is difficult to get global estimate of $\|\Omega\|_{L^{\infty}}$ from the following equation
\begin{align}\label{eq5}
\frac{d}{dt}\Omega + u\cdot\nabla\Omega = \nabla \times {\rm div}~\tau.
\end{align}
Motivated by \cite{2015Elgindi}, we can cancel $\nabla \times {\rm div}~\tau$ with the dissipation term $\Delta\tau$. Define
$$\Gamma = \mu\Omega-R\tau,~~~R=\Delta^{-1}{\rm curl}~{\rm div}.$$
Then, we deduce from \eqref{eq2} that
\begin{align}\label{eq6}
\frac{d}{dt}\Gamma + u\cdot\nabla\Gamma = aR\tau + RQ(\Omega,\tau) + [R,u\cdot\nabla]\tau\triangleq\sum_{i = 1}^3 F_i.
\end{align}
Different from \cite{2015Elgindi}, there is no damping phenomenon for $\Gamma$ or $\Omega$. It seems impossible to expect global existence even in small initial data case. However, the disappearance of $D(u)$ leads to exponential dissipation for $\tau$ in $H^1$, which is useful to estimate $\Gamma$ in $L^{\infty}$.

Assume that
\begin{align}\label{assumption}
\|\nabla u(t)\|_{L^{2}} \leq 4c\kappa,~~~~\|\Gamma(t)\|_{L^{\infty}} \leq 4ca\mu,
\end{align}
for any $t\in[0,T]$. By Proposition \ref{prop1} and the condition \eqref{condition1}, we deduce that $\|\nabla u(t)\|_{L^2}\leq 2c\kappa$ for any $t\in[0,T]$.
Then we focus on $\|\Gamma\|_{L^{\infty}}$. According to \eqref{eq6}, we obtain
\begin{align}\label{ineq17}
\|\Gamma\|_{L^{\infty}} \leq \|\Gamma_0\|_{L^{\infty}} + \sum_{i = 1}^3\int_0^t \|F_i\|_{L^{\infty}} ds.
\end{align}
From Lemma \ref{Lemma5}, we have
\begin{align}\label{F1}
\|F_1\|_{L^{\infty}} &\leq a\|\Delta_{-1}R\tau\|_{L^{\infty}} + a\|(Id - \Delta_{-1})R\tau\|_{L^{\infty}} \\ \notag
&\leq Ca\|\tau\|_{L^2} + Ca\|\tau\|_{B^{0}_{\infty,1}}\\ \notag
&\leq Ca\|\tau\|_{H^2}.
\end{align}
Applying Lemmas \ref{Lemma1}, \ref{Lemma5} and \ref{Lemma6}, we get
\begin{align}\label{F2}
\|F_2\|_{L^{\infty}} &\leq \|\Delta_{-1}RQ(\Omega,\tau)\|_{L^{\infty}} + \|(Id - \Delta_{-1})RQ(\Omega,\tau)\|_{L^{\infty}} \\ \nonumber
&\leq C\|Q(\Omega,\tau)\|_{L^2} + C\|Q(\Omega,\tau)\|_{B^{0}_{\infty,1}} \\ \nonumber
&\leq C\|\nabla u\|_{L^2}\|\tau\|_{H^2} + C\|\Omega\|_{B^{0}_{\infty,1}}\|\tau\|_{H^2} \\ \nonumber
&\leq C\|\nabla u\|_{L^2}\|\tau\|_{H^2} + C\|\Omega\|_{L^{\infty}}\ln(e+ \|u\|_{H^s})\|\tau\|_{H^2} + C\|\tau\|_{H^2}.
\end{align}
From Lemma \ref{Lemma2}, we obtain
\begin{align}\label{F3}
\|F_3\|_{L^{\infty}} &\leq C(\|\Omega\|_{L^2} + \|\Omega\|_{L^{\infty}})\|\tau\|_{H^2} \\ \nonumber
&\leq C\|\nabla u\|_{L^2}\|\tau\|_{H^2} +C_\mu\|\Gamma\|_{L^{\infty}}\|\tau\|_{H^2} +C_\mu\|R\tau\|_{L^{\infty}}\|\tau\|_{H^2} \\ \nonumber
&\leq C\|\nabla u\|_{L^2}\|\tau\|_{H^2} +C_\mu\|\Gamma\|_{L^{\infty}}\|\tau\|_{H^2} +C_\mu\|\tau\|^2_{H^2}.
\end{align}
Plugging \eqref{F1}-\eqref{F3} into \eqref{ineq17}, we deduce from \eqref{condition1}, \eqref{assumption} and Corollary \ref{coro1} that
\begin{align}\label{ineq18}
\|\Gamma\|_{L^{\infty}} &\leq \|\Gamma_0\|_{L^{\infty}} + C\int_0^t (1+a)\|\tau\|_{H^2}  + \|\nabla u\|_{L^2}\|\tau\|_{H^2} +C_\mu\|\Gamma\|_{L^{\infty}}\|\tau\|_{H^2} +C_\mu\|\tau\|^2_{H^2} ds\\ \notag
&~~~+\int_0^t C_\mu\|\tau\|^2_{H^2}\ln(e + \|u\|_{H^s}) +C_\mu\|\Gamma\|_{L^{\infty}}\|\tau\|_{H^2}\ln(e + \|u\|_{H^s})ds \\ \notag
& \leq \|\Gamma_0\|_{L^{\infty}} + C(1+a)(a^{-1}+(a\mu)^{-\frac 1 2})\|\tau_0\|_{H^1} + C_\mu\int_0^t\|\tau\|^2_{H^2}\ln(e + \|u\|_{H^s}) ds \\ \notag
&~~~+C_\mu\int_0^t\|\Gamma\|_{L^{\infty}}\|\tau\|_{H^2}\ln(e + \|u\|_{H^s}) ds.
\end{align}
By \eqref{condition2}, we get
\begin{align}\label{ineq19}
\|\Gamma\|_{L^{\infty}} \leq \frac 3 2 ca\mu+ C_\mu\int_0^t\|\tau\|^2_{H^2}\ln(e+ \|u\|_{H^s}) ds+C_\mu\int_0^t\|\Gamma\|_{L^{\infty}}\|\tau\|_{H^2}\ln(e+ \|u\|_{H^s}) ds,
\end{align}
where we using the condition $\|\tau_0\|_{H^1} \leq c^2\lambda$ with
\begin{align*}
\lambda\leq\min\{a^{2}\mu,(a\mu)^{\frac 3 2},a\mu,a^{\frac 1 2}\mu^{\frac 3 2}\}.
\end{align*}
According to Lemma \ref{Lemma5}, \eqref{assumption} and Corollary \ref{coro1}, we similarly obtain
\begin{align}\label{ineq21}
C\int_0^t \|\nabla u\|_{B^{0}_{\infty,\infty}}ds &\leq C\int_0^t \|\nabla u\|_{L^2}+\|\Omega\|_{L^\infty}ds  \\ \notag
&\leq \frac a 8 t+C_\mu\int_0^t\|\tau\|_{H^2} ds \\ \notag
&\leq \frac a 8 t+C_\mu(a^{-1}+(a\mu)^{-\frac 1 2})\|\tau_0\|_{H^1}  \\ \notag
&\leq \frac a 8 t+C.
\end{align}
By \eqref{condition2}, \eqref{ineq16'}, \eqref{ineq21} and Corollary \ref{coro1}, we deduce that
\begin{align}\label{ineq22}
\ln(e + \|(u,\tau)\|^2_{H^s} )&\leq Ce^{\frac a 8 t}[\ln(e + \|(u_0,\tau_0)\|^2_{H^s} ) +t(1+\mu^{-1})\\ \notag
&~~~+\int_0^t C_{\mu}(\|\tau\|^2_{L^{\infty}}+\|\nabla u\|_{L^{2}}\|\Omega\|_{L^{\infty}}) ds]  \\ \notag
&\leq Ce^{\frac a 8 t}[\ln(e + \|(u_0,\tau_0)\|^2_{H^s} ) +t(1+\mu^{-1}+c^2a)+c^4a\mu+c^3]  \\ \notag
&\leq Ce^{\frac a 4 t}[\ln(e +\|(u_0,\tau_0)\|^2_{H^s}) +(a\mu)^{-1}+a^{-1}+1+\mu]  \\ \notag
& = A_0 e^{\frac a 4 t},
\end{align}
where $A_0 = C[\ln(e +\|(u_0,\tau_0)\|^2_{H^s}) +(a\mu)^{-1}+a^{-1}+1+\mu].$
Plugging \eqref{ineq22} into \eqref{ineq19}, using \eqref{condition2} and applying Proposition \ref{prop1}, we obtain
\begin{align}\label{ineq23}
\|\Gamma\|_{L^{\infty}}
&\leq \frac 3 2 ca\mu+ C_\mu\int_0^t\|\tau\|^2_{H^2}A_0 e^{\frac a 4 t} ds+C_\mu\int_0^t\|\Gamma\|_{L^{\infty}}\|\tau\|_{H^2}A_0 e^{\frac a 4 t} ds \\ \notag
&\leq \frac 3 2 ca\mu+ C_\mu(\mu^{-1}+a^{-1})\|\tau_0\|^2_{H^1}A_0+(1+a^{\frac 1 2}\mu^{-\frac 1 2})\|\tau_0\|_{H^1}A_0 \\ \notag
&\leq \frac 3 2 ca\mu+(1+a^{\frac 1 2}\mu^{-\frac 1 2})\|\tau_0\|_{H^1}A_0 \\ \notag
&\leq 2ca\mu+(1+a^{\frac 1 2}\mu^{-\frac 1 2})\|\tau_0\|_{H^1}\ln(e +\|(u_0,\tau_0)\|^2_{H^s}) \\ \notag
&\leq 3ca\mu,
\end{align}
where we using the condition $\|\tau_0\|_{H^1} \leq \frac{c^2\lambda}{\ln(e +\|(u_0,\tau_0)\|^2_{H^s}) }$ with
\begin{align*}
\lambda&=\min\{(a\mu)^{\frac 3 2},a\mu,a^{\frac 1 2}\mu^{\frac 3 2},a^2\mu^2,a^2\mu,a,a^{\frac 3 2}\mu^{\frac 5 2},a^{\frac 1 2}\mu^{\frac 1 2}\}  \\
&=\min\{a^{\frac 1 2}\mu^{\frac 3 2},a^2\mu^2,a^2\mu,a,a^{\frac 3 2}\mu^{\frac 5 2},a^{\frac 1 2}\mu^{\frac 1 2}\}.
\end{align*}
According to Propositions \ref{prop1} and \ref{prop2}, we can deduce that $T^\ast=+\infty$.
We thus complete the proof of Theorem \ref{th1}.
\hfill$\Box$
\subsubsection{The inviscid Hooke model}
Taking $\psi =(g+1)\psi_{\infty}$ with $\psi_{\infty}=e^{-\frac 1 2 |q|^2}$ and $\nu=0,~a=2,~\mu=1$ in \eqref{eq1}, we obtain
\begin{align}\label{gequ}
\left\{
\begin{array}{ll}
\partial_tu+u\cdot\nabla u+\nabla P ={\rm div}~\tau,~~~~{\rm div}~u=0,\\[1ex]
\partial_t g + u\cdot\nabla g + \frac{1}{\psi_{\infty}}\nabla_q \cdot \big( \Omega q g \psi_{\infty}\big)-\Delta g = \frac{1}{\psi_{\infty}}\nabla_q \cdot \big( \nabla_qg \psi_{\infty}\big). \\[1ex]
\end{array}
\right.
\end{align}
Let $\langle q \rangle=\sqrt{1+q^2}$. Global well-posedness for the inviscid Hooke model \eqref{gequ} is considered in the following corollary. Firstly, we establish a new estimate of $\langle q\rangle^n\nabla^m_qg$ in $L^{\infty}(\mathcal{L}^2)$. Then,  we obtain the smallness of $\|\Omega\|_{L^{\infty}}$ under the condition \eqref{smallness} by virtue of the corresponding Ordroyd-B model \eqref{eq2}. Finally, we derive the global estimate for $\|u\|_{H^s} + \|\langle q\rangle g\|^2_{L^2(\mathcal{L}^2)} + \|\langle q\rangle\nabla_q g\|^2_{H^{s-1}(\mathcal{L}^2)}$, which implies the global existence of the inviscid Hooke model considered.
\begin{coro}\label{th1'}
Let $(u,g)$ be a strong solution of \eqref{gequ} with the initial data $(u_0,g_0)\in H^s\times H^s(\mathcal{L}^2)$ and $(\langle q\rangle g_0,\langle q\rangle \nabla_q g_0,\langle q\rangle \nabla^2_q g_0)\in L^{\infty}(\mathcal{L}^2)$. Let $\int_{\mathbb{R}^2} g_0\psi_{\infty}dq=0$ and $(u_0,\tau_0)$ satisfies the conditions in Theorem \ref{th1}. In addiction, if there exists some sufficiently small constant $\varepsilon$, which is not dependent on the initial data, such that
\begin{align}\label{smallness}
\|\tau_0\|_{B^0_{\infty,1}}+\|(u_0,\tau_0)\|_{L^2}\|\tau_0\|_{L^{\infty}}\leq \varepsilon,
\end{align}
then \eqref{gequ} admits a unique global strong solution $(u,g)\in C([0,\infty); H^s)\times C([0,\infty); H^s(\mathcal{L}^2))$.
\end{coro}
To begin with, we establish a new estimate of $\|\langle q\rangle^{n}\nabla^m_qg\|_{L^{\infty}(\mathcal{L}^2)}$ in the following lemma.
\begin{lemm}\label{Conservation}
Let $(u,g)$ be a strong solution of \eqref{gequ} with the initial data $(u_0,g_0)\in H^s\times H^s(\mathcal{L}^2)$ and $(\langle q\rangle g_0,\langle q\rangle \nabla_q g_0,\langle q\rangle \nabla^2_q g_0)\in L^{\infty}(\mathcal{L}^2)$. Let $\int_{\mathbb{R}^2} g_0\psi_{\infty}dq=0$. There exists positive constant $C$ such that
\begin{align*}
\|\langle q\rangle g\|_{L^{\infty}(\mathcal{L}^2)}+\|\langle q\rangle \nabla_q g\|_{L^{\infty}(\mathcal{L}^2)}+\|\langle q\rangle^2 \nabla_q g\|_{L^{\infty}(\mathcal{L}^2)}+\|\langle q\rangle \nabla^2_q g\|_{L^{\infty}(\mathcal{L}^2)} \leq C e^{Ct}.
\end{align*}
\end{lemm}
\begin{proof}
Firstly, we have $$\|g\|_{L^{\infty}(\mathcal{L}^2)}\leq\|g_0\|_{L^{\infty}(\mathcal{L}^2),}$$
by noticing that the term $\frac{1}{\psi_{\infty}}\nabla_q \cdot \big( \Omega q g \psi_{\infty}\big)$ would vanish in energy estimate since the antisymmetry of $\Omega$ and $\int_{\mathbb{R}^2} g\psi_{\infty}dq=0$. For more details, one can refer to \cite{HookeGlobal}.

For $\|\langle q\rangle g\|_{L^{\infty}(\mathcal{L}^2)}$, taking $\mathcal{L}^2$ inner product with $\langle q\rangle^2 g$ to $(\ref{gequ})_2$, we infer that
\begin{align}\label{ho1}
&\frac{1}{2}\frac{d}{dt}\|\langle q\rangle g\|^2_{\mathcal{L}^2} +\frac{1}{2} u\cdot\nabla\|\langle q\rangle g\|^2_{\mathcal{L}^2} -\frac{1}{2}\Delta\|\langle q\rangle g\|^2_{\mathcal{L}^2} + \|\langle q\rangle \nabla g\|^2_{\mathcal{L}^2} + \|\langle q\rangle \nabla_q g\|^2_{\mathcal{L}^2} \\ \notag
&= -\int_{\mathbb{R}^{2}}\frac{1}{\psi_{\infty}}\nabla_q \cdot \big( \Omega q g \psi_{\infty}\big)\langle q\rangle^2g\psi_{\infty}dq + \int_{\mathbb{R}^{2}} g^2\psi_{\infty} dq - \int_{\mathbb{R}^{2}} q^2 g^2\psi_{\infty} dq.
\end{align}
Since
\begin{align*}
\int_{\mathbb{R}^{2}}\frac{1}{\psi_{\infty}}\nabla_q \cdot \big( \Omega q g \psi_{\infty}\big)\langle q\rangle^2g\psi_{\infty}dq &= -\int_{\mathbb{R}^{2}} \Omega q g \psi_{\infty}\cdot\big(2qg + \langle q\rangle^2\nabla_qg\big)dq \\ \notag
&= -\frac{1}{2}\int_{\mathbb{R}^{2}} \Omega q\psi_{\infty}\langle q\rangle^2\nabla_qg^2dq \\ \notag
&= -\frac{1}{2}\int_{\mathbb{R}^{2}} \Omega^{ik} \big(\delta^i_k\langle q\rangle^2 + q_iq_k\big)g^2\psi_{\infty}dq \\ \notag
&=0,
\end{align*}
we can deduce from \eqref{ho1} that for any $p\geq2$, there holds
\begin{align}
&\frac{1}{p}\frac{d}{dt}\|\langle q\rangle g\|^p_{L^p(\mathcal{L}^2)} \leq C\|\langle q\rangle g\|^{p}_{L^p(\mathcal{L}^2)},
\end{align}
which implies $\|\langle q\rangle g\|_{L^{\infty}(\mathcal{L}^2)} \leq C e^{Ct}$.

Similarly, for $\|\langle q\rangle^{n} \nabla_q g\|_{L^{\infty}(\mathcal{L}^2)}$, $n\in\{1,2\}$, we have
\begin{align}\label{ho2}
&\frac{1}{2}\frac{d}{dt}\|\langle q\rangle^n\nabla_q g\|^2_{\mathcal{L}^2} +\frac{1}{2} u\cdot\nabla\|\langle q\rangle^n\nabla_q g\|^2_{\mathcal{L}^2} -\frac{1}{2}\Delta\|\langle q\rangle^n\nabla_q g\|^2_{\mathcal{L}^2}+\|\langle q\rangle^n\nabla\nabla_q g\|^2_{\mathcal{L}^2} \\ \notag
&=\int_{\mathbb{R}^{2}}\nabla_q (\frac{1}{\psi_{\infty}}\nabla_q \cdot(\nabla_q g\psi_{\infty}))\langle q\rangle^{2n}\nabla_qg\psi_{\infty}dq,
\end{align}
where we use the fact that
\begin{align*}
&\int_{\mathbb{R}^{2}}\nabla_q\big(\frac{1}{\psi_{\infty}}\nabla_q \cdot \big( \Omega q g \psi_{\infty}\big)\big)\langle q\rangle^{2n}\nabla_qg\psi_{\infty}dq \\ \notag
&= \int_{\mathbb{R}^{2}}\nabla^l_q\big(\Omega^{ik} q_k\nabla^i_qg - \Omega^{ik} q_kq_ig\big)\langle q\rangle^{2n}\nabla^l_qg\psi_{\infty}dq \\ \notag
&= \int_{\mathbb{R}^{2}}\big(\Omega^{ik} \delta^l_k\nabla^i_qg + \Omega^{ik} q_k\nabla^{il}_qg - \Omega^{ik} (\delta^l_kq_i+\delta^l_iq_k)g-\Omega^{ik} q_kq_i\nabla^l_qg\big)\langle q\rangle^{2n}\nabla^l_qg\psi_{\infty}dq\\ \notag
&=0.
\end{align*}
We deduce from Lemma \ref{Lemma7} that
\begin{align*}
&\int_{\mathbb{R}^{2}}\nabla_q (\frac{1}{\psi_{\infty}}\nabla_q \cdot(\nabla_q g\psi_{\infty}))\langle q\rangle^{2n}\nabla_qg\psi_{\infty}dq \\ \notag
&=\int_{\mathbb{R}^{2}}[\frac{1}{\psi_{\infty}}\nabla_q\cdot(\nabla_q \nabla_q g\psi_{\infty})-\nabla_qg]\langle q\rangle^{2n}\nabla_qg\psi_{\infty}dq\\ \notag
&=\int_{\mathbb{R}^{2}}\nabla_q\cdot(\nabla_q \nabla_q g\psi_{\infty})\langle q\rangle^{2n}\nabla_qgdq -\|\langle q\rangle^n \nabla_q g\|^2_{\mathcal{L}^2} \\ \notag
&=-2n\int_{\mathbb{R}^{2}}\nabla_q \nabla_q g\psi_{\infty}\langle q\rangle^{2(n-1)}q\nabla_qgdq-\|\langle q\rangle^n \nabla^2_q g\|^2_{\mathcal{L}^2} -\|\langle q\rangle^n \nabla_q g\|^2_{\mathcal{L}^2}\\ \notag
&\leq C\|\langle q\rangle^{n} \nabla_q g\|^2_{\mathcal{L}^2}.
\end{align*}
Plugging the above estimate into \eqref{ho2}, we infer that
\begin{align}
\frac{1}{2}\frac{d}{dt}\|\langle q\rangle^n\nabla_q g\|^2_{\mathcal{L}^2} +\frac{1}{2} u\cdot\nabla\|\langle q\rangle^n\nabla_q g\|^2_{\mathcal{L}^2} -\frac{1}{2}\Delta\|\langle q\rangle^n\nabla_q g\|^2_{\mathcal{L}^2} \leq C\|\langle q\rangle^{n} \nabla_q g\|^2_{\mathcal{L}^2}.
\end{align}
Therefore we deduce that for any $p\geq2$, there holds
\begin{align}
&\frac{1}{p}\frac{d}{dt}\|\langle q\rangle^n\nabla_q g\|^p_{L^p(\mathcal{L}^2)} \leq C\|\langle q\rangle^n\nabla_q g\|^p_{L^p(\mathcal{L}^2)},
\end{align}
which implies $\|\langle q\rangle^n\nabla_q g\|_{L^{\infty}(\mathcal{L}^2)} \leq C e^{Ct}$.

For $\|\langle q\rangle \nabla^2_q g\|_{L^{\infty}(\mathcal{L}^2)}$, we have
\begin{align}\label{ho3}
&\frac{1}{2}\frac{d}{dt}\|\langle q\rangle\nabla^2_q g\|^2_{\mathcal{L}^2} +\frac{1}{2} u\cdot\nabla\|\langle q\rangle\nabla^2_q g\|^2_{\mathcal{L}^2} -\frac{1}{2}\Delta\|\langle q\rangle\nabla^2_q g\|^2_{\mathcal{L}^2}+\|\langle q\rangle\nabla\nabla^2_q g\|^2_{\mathcal{L}^2}  \\ \notag
&=\int_{\mathbb{R}^{2}}\big[\nabla^2_q \big(\frac{1}{\psi_{\infty}}\nabla_q \cdot (\nabla_q g\psi_{\infty})\big)\big]\langle q\rangle^{2}\nabla^2_qg\psi_{\infty}dq,
\end{align}
where we use the fact that
\begin{align*}
&\int_{\mathbb{R}^{2}}\nabla^2_q\big(\frac{1}{\psi_{\infty}}\nabla_q \cdot \big( \Omega q g \psi_{\infty}\big)\big)\langle q\rangle^2\nabla^2_qg\psi_{\infty}dq \\ \notag
&= \int_{\mathbb{R}^{2}}\nabla^{lm}_q\big(\Omega^{ik} q_k\nabla^i_qg - \Omega^{ik} q_kq_ig\big)\langle q\rangle^2\nabla^{lm}_qg\psi_{\infty}dq \\ \notag
&= \int_{\mathbb{R}^{2}}\nabla^m_q\big(\Omega^{il}\nabla^i_qg + \Omega^{ik} q_k\nabla^{il}_qg - \big(\Omega^{il}q_i+\Omega^{lk}q_k)g-\Omega^{ik} q_kq_i\nabla^l_qg\big)\langle q\rangle^2\nabla^{lm}_qg\psi_{\infty}dq \\ \notag
&= \int_{\mathbb{R}^{2}}\big(\Omega^{il}\nabla^{im}_qg + \Omega^{im}\nabla^{il}_qg + \Omega^{ik} q_k\nabla^{ilm}_qg- \Omega^{ik} q_kq_i\nabla^{lm}_qg - \big(\Omega^{ml}+\Omega^{lm}\big)g \\ \notag
&~~~-\big(\Omega^{il}q_i+\Omega^{lk}q_k\big)\nabla^m_qg - \big(\Omega^{im}q_i+\Omega^{mk} q_k\big)\nabla^l_qg \big)\langle q\rangle^2\nabla^{lm}_qg\psi_{\infty}dq \\ \notag
&=0.
\end{align*}
Moreover, we can deduce that
\begin{align*}
\int_{\mathbb{R}^{2}}\big[\nabla^2_q \big(\frac{1}{\psi_{\infty}}\nabla_q \cdot (\nabla_q g\psi_{\infty})\big)\big]\langle q\rangle^{2}\nabla^2_qg\psi_{\infty}dq &=  \int_{\mathbb{R}^{2}}\big[\frac{1}{\psi_{\infty}}\nabla_q\cdot(\nabla_q \nabla^2_q g\psi_{\infty})-\nabla^2_qg\big]\langle q\rangle^{2}\nabla^2_qg\psi_{\infty}dq\\ \notag
&=\|\nabla_q^2 g\|^2_{\mathcal{L}^2}-\|\langle q\rangle \nabla_q\nabla^2_q g\|^2_{\mathcal{L}^2} -2\|\langle q\rangle \nabla^2_q g\|^2_{\mathcal{L}^2}.
\end{align*}
Plugging the above equality into \eqref{ho3}, we infer that
\begin{align}
&\frac{1}{p}\frac{d}{dt}\|\langle q\rangle\nabla_q^2 g\|^p_{L^p(\mathcal{L}^2)} \leq C\|\langle q\rangle\nabla_q^2 g\|^p_{L^p(\mathcal{L}^2)},
\end{align}
which implies $\|\langle q\rangle\nabla_q^2 g\|_{L^{\infty}(\mathcal{L}^2)} \leq C e^{Ct}$. We thus complete the proof of Lemma \ref{Conservation}.
\end{proof}
By virtue of Theorem \ref{th1}, we obtain the global existence of $u$. The following lemma is about the global existence of $g$.
\begin{lemm}\label{Omega}
Let $(u,g)$ be a strong solution of \eqref{gequ} with the initial data $(u_0,g_0)\in H^s\times H^s(\mathcal{L}^2)$. Let $\int_{\mathbb{R}^2} g_0\psi_{\infty}dq=0$ and $(u_0,\tau_0)$ satisfies the conditions in Theorem \ref{th1}. Then for any $\sigma>0$, there exist positive constant $\varepsilon$ small enough such that if
$$
\|\tau_0\|_{B^0_{\infty,1}}+ \|(u_0,\tau_0)\|_{L^2}\|\tau_0\|_{L^{\infty}}\leq\varepsilon,
$$
then there holds
$$\|\Omega\|_{L^{\infty}}\leq \sigma.$$
\end{lemm}
\begin{proof}
By virtue of Lemma \ref{Lemma5}, we deduce that
\begin{align}
\|\Omega\|_{L^{\infty}} \leq \|\Gamma\|_{L^{\infty}} + \|R\tau\|_{L^{\infty}} \leq \|\Gamma\|_{L^{\infty}} + C\|\tau\|_{B^0_{\infty,1}}.
\end{align}
It's follows from the proofs of Theorem \ref{th1} that
\begin{align*}
\|\Gamma\|_{L^{\infty}} \leq C\varepsilon \leq \frac{\sigma}{2},
\end{align*}
provided small enough $\varepsilon\leq\frac{\sigma}{4C}$.

To prove $C\|\tau\|_{B^0_{\infty,1}}\leq\frac{\sigma}{2}$, we now focus on high frequency estimate for $\tau$.
Applying $\Delta_j$ to $(\ref{eq2})_2$ with $j\geq 0$ yields
\begin{align}\label{taueq}
\partial_t\Delta_j\tau + 2\Delta_j\tau+\Delta_jQ(\Omega,\tau)=\Delta\Delta_j\tau -\Delta_j(u\cdot\nabla\tau).
\end{align}
Therefore
\begin{align}\label{ho4}
\Delta_j\tau = e^{-t(2+\Delta)}\Delta_j\tau_0 - \int_0^t e^{-2(t-s)}e^{(t-s)\Delta}\big(\Delta_jQ(\Omega,\tau) + \Delta_j(u\cdot\nabla\tau)\big)ds.
\end{align}
According to \eqref{ho4}, we deduce that
\begin{align}\label{ho5}
\sum_{j\geq 0}\|\Delta_{j}\tau\|_{L^{\infty}}
&\leq C\|\tau_0\|_{B^0_{\infty,1}}+\sum_{j\geq 0}\|\int_0^t e^{-2(t-s)}e^{(t-s)\Delta}\Delta_jQ(\Omega,\tau)ds\|_{L^{\infty}} \\ \notag
&~~~+\sum_{j\geq 0}\|\int_0^t e^{-2(t-s)}e^{(t-s)\Delta}\Delta_j(u\cdot\nabla\tau)ds\|_{L^{\infty}}.
\end{align}
According to Lemma \ref{Lemma1'} with $j\geq 0$, we infer that
\begin{align}\label{tauesti1}
\sum_{j\geq 0}\|\int_0^t e^{-2(t-s)}e^{(t-s)\Delta}\Delta_jQ(\Omega,\tau)ds\|_{L^{\infty}}
&\leq C\sum_{j\geq 0}\int_0^t e^{-2^{2j}(t-s)}\|\Delta_jQ(\Omega,\tau)\|_{L^{\infty}} ds \\ \notag
&\leq C\sum_{j\geq 0}\int_0^t e^{-2^{2j}(t-s)}2^{j}\|\Delta_jQ(\Omega,\tau)\|_{L^2} ds\\ \notag
&\leq C\sum_{j\geq 0}\int_0^t e^{-2^{2j}(t-s)}2^{\frac{3}{2}j}\|\nabla u\|_{L^2}\|\tau\|_{L^{\infty}} ds.
\end{align}
Similarly, by virtue of ${\rm div}~u=0$, we have
\begin{align}\label{tauesti2}
\sum_{j\geq 0}\|\int_0^t e^{-2(t-s)}e^{(t-s)\Delta}\Delta_j(u\cdot\nabla\tau)ds\|_{L^{\infty}}&\leq C\sum_{j\geq 0}\int_0^t e^{-2^{2j}(t-s)}2^{j}\|\Delta_j(u\otimes\tau)\|_{L^{\infty}}ds \\ \notag
&\leq C\sum_{j\geq 0}\int_0^t e^{-2^{2j}(t-s)}2^{\frac{3}{2}j}\|u\otimes\tau\|_{L^4}ds \\ \notag
&\leq C\sum_{j\geq 0}\int_0^t e^{-2^{2j}(t-s)}2^{\frac{3}{2}j}\|u\|_{H^1}\|\tau\|_{L^{\infty}}ds.
\end{align}
Notice that
\begin{align}
\|\tau\|_{L^{\infty}}\leq \|\tau_0\|_{L^{\infty}},
\end{align}
and
\begin{align}\label{sum}
\sup_{t\geq 0} \mathop{\sum}\limits_{j\in\mathcal{N}} \int_0^t e^{-2^{2j}(t-s)}2^{\frac{3}{2}j} ds \leq C.
\end{align}
According to \eqref{ho5}-\eqref{sum} and Proposition \ref{prop1}, we deduce that
\begin{align}
\|\tau\|_{B^0_{\infty,1}} &\leq \|\Delta_{-1}\tau\|_{L^{\infty}}+\sum_{j\geq 0}\|\Delta_{j}\tau\|_{L^{\infty}} \\ \notag
&\leq C\|\tau\|_{L^{\infty}}+C\|\tau_0\|_{B^0_{\infty,1}} + C(\|\nabla u\|_{L^{\infty}[0,T);L^2)}+\|(u_0,\tau_0)\|_{L^2})\|\tau_0\|_{L^{\infty}}\\ \notag
&\leq C(\varepsilon+\varepsilon^2)\leq\frac{\sigma}{2}.
\end{align}
We thus complete the proof of Lemma \ref{Omega}.
\end{proof}
{\bf The proof of Corollary \ref{th1'} :}  \\
By virtue of the lemmas above, we finally obtain global well-posedness of $(u,g)$.
Taking $L^2(\mathcal{L}^2)$ inner product with $g$ to  $(\ref{gequ})_2$, we deduce that
\begin{align}\label{ho6}
\frac 1 2\frac{d}{dt}\|g\|^2_{L^2(\mathcal{L}^2)} + \|\nabla g\|^2_{L^2(\mathcal{L}^2)} + \|\nabla_qg\|^2_{L^2(\mathcal{L}^2)} = 0.
\end{align}
Applying $\Lambda^s$ to $(\ref{gequ})_2$ and taking $L^2(\mathcal{L}^2)$ inner product with $\Lambda^sg$, we deduce from Lemma \ref{Lemma4} that
\begin{align}
&\frac 1 2\frac{d}{dt}\|\Lambda^sg\|^2_{L^2(\mathcal{L}^2)} + \|\Lambda^{s+1} g\|^2_{L^2(\mathcal{L}^2)} + \|\nabla_q\Lambda^sg\|^2_{L^2(\mathcal{L}^2)}\\ \notag
&\leq C_\varepsilon\|u\|^2_{L^{\infty}}\|\Lambda^sg\|^2_{L^2(\mathcal{L}^2)} + C_\varepsilon\|u\|^2_{H^s}\|g\|^2_{L^{\infty}(\mathcal{L}^2)}  + \varepsilon\|\Lambda^{s+1}g\|^2_{L^2(\mathcal{L}^2)}\\ \notag
&~~~+\int_{\mathbb{R}^{2}\times \mathbb{R}^{2}}\Lambda^s\big(\frac{1}{\psi_{\infty}}\nabla_q \cdot \big( \Omega q g \psi_{\infty}\big)\big)\Lambda^sg\psi_{\infty}dqdx.
\end{align}
According to Lemma \ref{Lemma7}, we deduce that
\begin{align*}
&\int_{\mathbb{R}^{2}\times \mathbb{R}^{2}}\Lambda^s\big(\frac{1}{\psi_{\infty}}\nabla_q \cdot \big( \Omega q g \psi_{\infty}\big)\big)\Lambda^sg\psi_{\infty}dqdx \\ \notag
&\leq \int_{\mathbb{R}^{2}}\|\Lambda^{s-1}\big(\Omega^{ik}q_i\nabla^k_qg + \Omega^{ik}q_iq_kg\big)\|_{L^2}\|\Lambda^{s+1}g\|_{L^2}\psi_{\infty}dq \\ \notag
&\leq  C_{\varepsilon}\|\Lambda^{s-1}\big(\Omega^{ik}q_i\nabla^k_qg + \Omega^{ik}q_iq_kg\big)\|^2_{L^2(\mathcal{L}^2)} + \varepsilon\|\Lambda^{s+1}g\|^2_{L^2(\mathcal{L}^2)} \\ \notag
&\leq C_{\varepsilon}\|\Omega\|^2_{L^{\infty}}\|\langle q\rangle\nabla_q\Lambda^{s-1}g\|^2_{L^2(\mathcal{L}^2)} + C_{\varepsilon}\|u\|^2_{H^s}\|\langle q\rangle\nabla_qg\|^2_{L^{\infty}(\mathcal{L}^2)} + \varepsilon\|\Lambda^{s+1}g\|^2_{L^2(\mathcal{L}^2)},
\end{align*}
which implies that
\begin{align}\label{ineq1'}
\frac{d}{dt}\|\Lambda^sg\|^2_{L^2(\mathcal{L}^2)} + \|\Lambda^{s+1} g\|^2_{L^2(\mathcal{L}^2)} + \|\nabla_q\Lambda^sg\|^2_{L^2(\mathcal{L}^2)}
&\leq C\|u\|^2_{L^{\infty}}\|\Lambda^sg\|^2_{L^2(\mathcal{L}^2)}\\ \notag
&~~~+ C\|u\|^2_{H^s}(\|g\|^2_{L^{\infty}(\mathcal{L}^2)}+\|\langle q\rangle\nabla_qg\|^2_{L^{\infty}(\mathcal{L}^2)})\\ \notag
&~~~+ C\|\Omega\|^2_{L^{\infty}}\|\langle q\rangle\nabla_qg\|^2_{H^{s-1}(\mathcal{L}^2)}.
\end{align}
The appearance of the term $\|\langle q\rangle\nabla_qg\|^2_{H^{s-1}(\mathcal{L}^2)}$ force us to consider mixed derivative estimates which have been used in \cite{HookeGlobal} and \cite{2017Global}. 

Applying $\Lambda^m$ to  $(\ref{gequ})_2$ with $m\in\{0,s\}$ and taking $L^2(\mathcal{L}^2)$ inner product with $\langle q\rangle^2\Lambda^mg$, we infer from Lemmas \ref{Lemma4} and \ref{Lemma7} that
\begin{align}
&\frac{d}{dt}\|\langle q\rangle g\|^2_{L^2(\mathcal{L}^2)} + \|\langle q\rangle \nabla g\|^2_{L^2(\mathcal{L}^2)} + \|\langle q\rangle \nabla_qg\|^2_{L^2(\mathcal{L}^2)} \leq C\|g\|^2_{L^2(\mathcal{L}^2)},
\end{align}
and
\begin{align}
&\frac{d}{dt}\|\langle q\rangle\Lambda^sg\|^2_{L^2(\mathcal{L}^2)} + \|\langle q\rangle\Lambda^{s+1} g\|^2_{L^2(\mathcal{L}^2)} + \|\langle q\rangle\nabla_q\Lambda^sg\|^2_{L^2(\mathcal{L}^2)}\\ \notag
&\leq  C\|u\|^2_{L^{\infty}}\|\langle q\rangle\Lambda^sg\|^2_{L^2(\mathcal{L}^2)}+ C\|\Lambda^sg\|^2_{L^2(\mathcal{L}^2)} \\ \notag
&~~~+ C\|u\|^2_{H^s}\big(\|g\|^2_{L^{\infty}(\mathcal{L}^2)} + \|\langle q\rangle^2\nabla_qg\|^2_{L^{\infty}(\mathcal{L}^2)}\big)\\ \notag
&~~~+ C\|\Omega\|^2_{L^{\infty}}\|\langle q\rangle\nabla^2_q\Lambda^{s-1}g\|^2_{L^2(\mathcal{L}^2)} .
\end{align}

Applying $\nabla_q\Lambda^m$ to  $(\ref{gequ})_2$ with $m\in\{0,s-1\}$ and taking $L^2(\mathcal{L}^2)$ inner product with $\langle q\rangle^2\nabla_q\Lambda^mg$, we infer from Lemmas \ref{Lemma4} and \ref{Lemma7} that
\begin{align}
&\frac{d}{dt}\|\langle q\rangle\nabla_q g\|^2_{L^2(\mathcal{L}^2)} + \|\langle q\rangle \nabla_q \nabla g\|^2_{L^2(\mathcal{L}^2)} + \|\langle q\rangle \nabla^2_qg\|^2_{L^2(\mathcal{L}^2)} \leq C\|\langle q\rangle \nabla_q g\|^2_{L^2(\mathcal{L}^2)},
\end{align}
and
\begin{align}
&\frac{d}{dt}\|\langle q\rangle\nabla_q\Lambda^{s-1}g\|^2_{L^2(\mathcal{L}^2)} + \|\langle q\rangle\nabla_q\Lambda^s g\|^2_{L^2(\mathcal{L}^2)} + \|\langle q\rangle\nabla^2_q\Lambda^{s-1}g\|^2_{L^2(\mathcal{L}^2)}\\ \notag
&\leq C\|\langle q\rangle\nabla_q\Lambda^{s-1}g\|^2_{L^2(\mathcal{L}^2)} + C\|u\|^2_{L^{\infty}}\|\langle q\rangle\nabla_q\Lambda^{s-1}g\|^2_{L^2(\mathcal{L}^2)} \\ \notag
&~~~+ C\|u\|^2_{H^s}\big(\|g\|^2_{L^{\infty}(\mathcal{L}^2)} + \|\nabla_q g\|^2_{L^{\infty}(\mathcal{L}^2)} + \|\langle q\rangle\nabla^2_qg\|^2_{L^{\infty}(\mathcal{L}^2)}\big)\\ \notag
&~~~+ C\|\Omega\|^2_{L^{\infty}}\big(\|\nabla_q\Lambda^{s-1}g\|^2_{L^2(\mathcal{L}^2)} + \|\langle q\rangle\nabla^2_q\Lambda^{s-1}g\|^2_{L^2(\mathcal{L}^2)}\big).
\end{align}

For $u$, we have the following estimate
\begin{align}\label{ineq2'}
\frac{d}{dt}\|u\|^2_{H^s} \leq C(\|\nabla u\|_{L^{\infty}}+1)\|u\|^2_{H^s} + C\|g\|^2_{H^{s+1}(\mathcal{L}^2)}.
\end{align}
According to \eqref{ho6}-\eqref{ineq2'}, Lemma \ref{Omega} with $\sigma$ small enough and Gronwall's inequality, we deduce that
\begin{align}
&\|u\|_{H^s} + \|\langle q\rangle g\|^2_{H^s(\mathcal{L}^2)} + \|\langle q\rangle\nabla_q g\|^2_{H^{s-1}(\mathcal{L}^2)} \\ \notag
&\leq C\big(\|u_0\|^2_{H^s} + \|\langle q\rangle g_0\|^2_{H^s(\mathcal{L}^2)} + \|\langle q\rangle\nabla_q g_0\|^2_{H^{s-1}(\mathcal{L}^2)}\big)e^{Ct+C\int_{0}^{t}\|u\|^2_{L^{\infty}}+\|\nabla u\|_{L^{\infty}}dt'}.
\end{align}
By virtue of Theorem~\ref{th1}, we finish the proof of Corollary \ref{th1'}.
\hfill$\Box$

\begin{rema}
The main difficulty for the global estimates of  $(\ref{gequ})_2$ is that once we stop the growth of regularity in $x$ by $-\Delta g$, we can not stop the growth of power $\langle q\rangle$ by $\mathcal{L}g=-\frac{1}{\psi_{\infty}}\nabla_q \cdot \big( \nabla_qg \psi_{\infty}\big)$ at the same time. It is worth mentioning that the estimate of $\|\langle q\rangle\nabla_qg\|^2_{H^{s-1}(\mathcal{L}^2)}$ instead of $\|\langle q\rangle\nabla_qg\|^2_{H^{s}(\mathcal{L}^2)}$ enable us to stop the growth of power $\langle q\rangle$ caused by the term $\frac{1}{\psi_{\infty}}\nabla_q \cdot \big( \Omega q g \psi_{\infty}\big)$. As a result, we obtain the a prior estimate of $\|\langle q\rangle g\|^2_{H^s(\mathcal{L}^2)}$ and $\|\langle q\rangle\nabla_qg\|^2_{H^{s-1}(\mathcal{L}^2)}$.
\end{rema}
\begin{rema}
The estimate of $\langle q\rangle^{n}\nabla^m_qg$ in $L^{\infty}(\mathcal{L}^2)$ and the smallness of $\|\Omega\|_{L^{\infty}}$ are significant in the proof of Corollary \ref{th1'}. Global existence of \eqref{gequ} for arbitrary initial data and global existence of the noncorotation inviscid Hooke model are interesting problems. We are going to study these problems in the future.
\end{rema}
\section{Global solutions for the co-rotation case in critical Besov space}
In this section, we are concerned with global solutions to the co-rotation inviscid Oldroyd-B model in critical Besov space. We
divide it into three subsections to prove Theorem \ref{th13}.
\subsection{Energy estimates}
From now on, we prove the boundness of smooth solutions for \eqref{eq2} in the following propositions.
\begin{prop}\label{7lemm1}
Suppose $(u,\tau)$ is a smooth solution to \eqref{eq2} with $u_0\in H^1$ and $\tau_0\in H^1\cap L^{\infty}$. Then we obtain
\begin{align}\label{7ineq0}
\|\nabla u\|^2_{L^2}+\|\nabla\tau\|^2_{L^2} \leq (\|\nabla u_0\|^2_{L^2}+\|\nabla\tau_0\|^2_{L^2}) e^{\frac{C}{a\mu}+\frac{C}{a\mu}\|\tau_0\|^2_{L^{\infty}}+C\mu^{-2}\|\tau_0\|^2_{L^2}}.
\end{align}
Moreover, we get
\begin{align}\label{7ineq1}
\|(u,\tau)\|^2_{H^1} \leq H_0,
\end{align}
where $H_0=\|(u_0,\tau_0)\|^2_{H^1} e^{\frac{C}{a\mu}+\frac{C}{a\mu}\|\tau_0\|^2_{L^{\infty}}+C\mu^{-2}\|\tau_0\|^2_{L^2}}$.
\end{prop}
\begin{proof}
Taking the $L^2$ inner product with $-\Delta\tau$ to $(\ref{eq2})_2$ and using Lemma \ref{Lemma3}, we have
\begin{align}\label{7ineq2}
\frac{1}{2}\frac{d}{dt}\|\nabla\tau\|^2_{L^2} + a\|\nabla\tau\|^2_{L^2} + \mu\|\nabla^2\tau\|^2_{L^2}
&=\langle u\cdot\nabla\tau,\Delta\tau\rangle + \langle Q(\Omega,\tau),\Delta\tau\rangle \\ \notag
&\leq C\|\nabla u\|_{L^2}\|\nabla\tau\|^2_{L^4} + C\|\Omega\|_{L^2}\|\tau\|_{L^{\infty}}\|\Delta\tau\|_{L^2} \\ \notag
&\leq C_\mu\|\nabla u\|^2_{L^2}(\|\nabla\tau\|^2_{L^2}+\|\tau\|^2_{L^{\infty}})+\frac{2\mu}{3}\|\nabla^2\tau\|^2_{L^2}.
\end{align}
Consider $\tilde{\tau}^{ij}=\tau^{ij}e^{\frac{a}{2}t}$, we infer from \eqref{7ineq2} that
\begin{align}\label{7ineq3}
\frac{1}{2}\frac{d}{dt}\|\nabla\tilde{\tau}\|^2_{L^2} + \mu\|\nabla^2\tilde{\tau}\|^2_{L^2} \leq C_\mu\|\nabla u\|^2_{L^2}(\|\nabla\tilde{\tau}\|^2_{L^2}+\|\tilde{\tau}\|^2_{L^{\infty}})+\frac{2\mu}{3}\|\nabla^2\tilde{\tau}\|^2_{L^2}.
\end{align}
Taking the $L^2$ inner product with $-\Delta u$ to $(\ref{eq2})_1$, we have
\begin{align}\label{7ineq4}
\frac{1}{2}\frac{d}{dt}\|\nabla u\|^2_{L^2} \leq C_\mu e^{-at}\|\nabla u\|^2_{L^2}+\frac{\mu}{3}\|\nabla^2 \tilde{\tau}\|^2_{L^2}.
\end{align}
Combining $(\ref{7ineq3})$ and $(\ref{7ineq4})$, we deduce that
\begin{align}\label{7ineq5}
&\frac{1}{2}\frac{d}{dt}(\|\nabla u\|^2_{L^2}+\|\nabla\tilde{\tau}\|^2_{L^2}) \leq C_\mu\|\nabla u\|^2_{L^2}(e^{-at}+\|\nabla\tilde{\tau}\|^2_{L^2}+\|\tilde{\tau}\|^2_{L^{\infty}}).
\end{align}
Applying Gronwall's inequality to \eqref{7ineq5} and using Propositions \ref{7lemm1'}, \ref{prop1}, we obtain
\begin{align}\label{7ineq6}
\|\nabla u\|^2_{L^2}+\|\nabla\tilde{\tau}\|^2_{L^2} &\leq (\|\nabla u_0\|^2_{L^2}+\|\nabla\tau_0\|^2_{L^2}) e^{C_\mu\int_0^te^{-as}+\|\nabla\tilde{\tau}\|^2_{L^2}+\|\tilde{\tau}\|^2_{L^{\infty}}ds}\\ \notag
&\leq (\|\nabla u_0\|^2_{L^2}+\|\nabla\tau_0\|^2_{L^2}) e^{\frac{C}{a\mu}+\frac{C}{a\mu}\|\tau_0\|^2_{L^{\infty}}+C_\mu\int_0^te^{2as}\|\nabla\tau\|^2_{L^2}ds}\\
\notag
&\leq (\|\nabla u_0\|^2_{L^2}+\|\nabla\tau_0\|^2_{L^2}) e^{\frac{C}{a\mu}+\frac{C}{a\mu}\|\tau_0\|^2_{L^{\infty}}+C\mu^{-2}\|\tau_0\|^2_{L^2}}.
\end{align}
Combining \eqref{7ineq6} and \eqref{ineq1}, we finish the proof of Proposition \ref{7lemm1}.
\end{proof}

Then, we consider time integrability of $\tau$. The estimate with exponential weight in the following proposition is the key to estimating $\Gamma$ in $B^0_{\infty,1}$.
\begin{prop}\label{7lemm2}
Suppose $(u,\tau)$ is a smooth solution to \eqref{eq2} with $u_0\in H^1$ and $\tau_0\in H^1\cap L^{\infty}$. Then we obtain
\begin{align}
\|e^{\frac{a}{2}s}\tau\|_{\tilde{L}^{2}_T(B^{\frac{1}{2}}
	_{\infty,\infty})}
\leq B_0,
\end{align}
where $B_0=C(a^{-\frac{1}{2}}+\mu^{-\frac{1}{2}})\|\tau_0\|_{L^4}+ C\mu^{-1}a^{-\frac{1}{2}}H^{\frac 1 2}_0\|\tau_0\|_{L^{\infty}}$.
\end{prop}
\begin{proof}
To prove $\|e^{\frac{a}{2}s}\tau\|_{\tilde{L}^{2}_T(B^{\frac{1}{2}}_{\infty,\infty})} \leq B_0$, we now focus on high frequency estimate for $\tau$.
Set $\tilde{\tau}=\tau e^{at}$, we infer from $(\ref{eq2})_2$ that
\begin{align}\label{7ineq7}
\partial_t\tilde{\tau} + u\cdot\nabla\tilde{\tau} + Q(\Omega,\tilde{\tau}) = \mu\Delta\tilde{\tau}.
\end{align}
Applying $\Delta_j$ to \eqref{7ineq7} with $j\geq 0$, we obtain
\begin{align}\label{7ineq8}
\partial_t\Delta_j\tilde{\tau} - \mu\Delta_j\Delta\tilde{\tau} =-\Delta_j{\rm div}~(u\tilde{\tau})-\Delta_jQ(\Omega,\tilde{\tau}),
\end{align}
which implies that
\begin{align}\label{7ineq9}
\Delta_je^{\frac{a}{2}t}\tau =e^{\mu t\Delta-\frac{a}{2}t}\Delta_j\tau_0-\int_0^te^{\mu(t-s)\Delta-\frac{a}{2}t}\big(\Delta_j{\rm div}~(u\tilde{\tau})+\Delta_jQ(\Omega,\tilde{\tau})\big)ds.
\end{align}
From \eqref{7ineq9}, we obtain
\begin{align}\label{7ineq9'}
(\int_0^t 2^{j}\|\Delta_j e^{\frac{a}{2}s}\tau\|^2_{L^{\infty}} ds)^{\frac{1}{2}}
&\leq (\int_0^t 2^{j}\|e^{-\frac{a}{2}s}e^{\mu s\Delta}\Delta_j\tau_0\|^2_{L^{\infty}} ds)^{\frac{1}{2}}\\ \notag
&~~~+(\int_0^t 2^{j}\|\int_0^s e^{-\frac{a}{2}s}e^{\mu(s-t')\Delta}\Delta_j{\rm div}~(u\tilde{\tau})dt'\|^2_{L^{\infty}} ds)^{\frac{1}{2}}\\ \notag
&~~~+(\int_0^t 2^{j}\|\int_0^s e^{-\frac{a}{2}s}e^{\mu(s-t')\Delta}\Delta_jQ(\Omega,\tilde{\tau})dt'\|^2_{L^{\infty}} ds)^{\frac{1}{2}}.
\end{align}
According to Lemma \ref{Lemma1'} with $j\geq 0$ and Proposition \ref{7lemm1'}, we deduce that
\begin{align}\label{7ineq10}
(\int_0^t 2^{j}\|e^{-\frac{a}{2}s}e^{\mu s\Delta}\Delta_j\tau_0\|^2_{L^{\infty}} ds)^{\frac{1}{2}} &\leq C(\int_0^t e^{-\mu2^{2j}s-as}2^{2j}\|\Delta_j\tau_0\|^2_{L^{4}}ds )^{\frac{1}{2}}\\ \notag
 &\leq C\|\tau_0\|_{L^{4}} (\int_0^t e^{-\mu2^{2j}s-as}2^{2j}ds)^{\frac{1}{2}} \\ \notag
 &\leq C\mu^{-\frac{1}{2}}\|\tau_0\|_{L^{4}} .
\end{align}
By virtue of Minkowski's inequality, Lemma \ref{Lemma1'} and Proposition \ref{7lemm1}, we infer that
\begin{align}\label{7ineq11}
&(\int_0^t 2^{j}\|\int_0^s e^{-\frac{a}{2}s}e^{\mu(s-t')\Delta}\big(\Delta_j{\rm div}~(u\tilde{\tau})+\Delta_jQ(\Omega,\tilde{\tau})\big)dt'\|^2_{L^{\infty}} ds)^{\frac{1}{2}} \\ \notag
&\leq C(\int_0^t e^{-as}(\int_0^s 2^{\frac{ j}{2}}e^{-\mu2^{2j}(s-t')}\big(\|\Delta_j{\rm div}~(u\tilde{\tau})\|_{L^{\infty}}+\|\Delta_jQ(\Omega,\tilde{\tau})\|_{L^{\infty}}\big)dt')^{2}ds )^{\frac{1}{2}} \\ \notag
&\leq C(\int_0^te^{-as}(\int_0^s 2^{\frac{j}{2}}e^{-\mu2^{2j}(s-t')}\big(2^{\frac{3}{2}j}\|u\|_{L^4}\|\tilde{\tau}\|_{L^{\infty}}+2^{j}\|\Omega\|_{L^2}\|\tilde{\tau}\|_{L^{\infty}}\big)dt')^2ds)^{\frac{1}{2}}\\ \notag
&\leq
C(\int_0^te^{-as}(\int_0^s 2^{2j}e^{-\mu2^{2j}(s-t')}\|u\|_{H^1}\|\tilde{\tau}\|_{L^{\infty}}dt')^2ds)^{\frac{1}{2}}\\ \notag
&\leq CH^{\frac 1 2}_0\|\tau_0\|_{L^{\infty}}(\int_0^te^{-as}(\int_0^s2^{2j}e^{-\mu2^{2j}(s-t')}dt')^2ds)^{\frac{1}{2}}\\ \notag
&=C\mu^{-1}H^{\frac 1 2}_0\|\tau_0\|_{L^{\infty}}(\int_0^te^{-as}(1-e^{-\mu 2^{2j}s})^2ds)^{\frac{1}{2}}\\ \notag
&\leq C\mu^{-1}a^{-\frac{1}{2}}H^{\frac 1 2}_0\|\tau_0\|_{L^{\infty}}.
\end{align}
According to \eqref{7ineq9'})-\eqref{7ineq11}, we deduce that
\begin{align}\label{7ineq12}
\|e^{\frac{a}{2}t}\tau\|_{\tilde{L}^2(B^{\frac{1}{2}}_{\infty,\infty})} &\leq (\int_0^t 2^{-1}\|\Delta_{-1} e^{\frac{a}{2}s}\tau\|^2_{L^{\infty}} ds)^{\frac{1}{2}}+\sup_{j\geq 0}(\int_0^t 2^{j}\|\Delta_j e^{\frac{a}{2}s}\tau\|^2_{L^{\infty}} ds)^{\frac{1}{2}}  \\ \notag
&\leq C(a^{-\frac{1}{2}}+\mu^{-\frac{1}{2}})\|\tau_0\|_{L^4}+ C\mu^{-1}a^{-\frac{1}{2}}H^{\frac 1 2}_0\|\tau_0\|_{L^{\infty}}.
\end{align}
We thus complete the proof of Proposition \ref{7lemm2}.
\end{proof}
\subsection{Local well-posedness}
\begin{prop}\label{5prop1}
Let $(u_0,\tau_0)\in B^1_{\infty,1}\times B^0_{\infty,1}$. There exists a time $T>0$ such that $(\ref{eq2})$ has a solution $(u,\tau)\in L^{\infty}([0,T);B^1_{\infty,1}\times B^0_{\infty,1})$.
\end{prop}
\begin{proof}
Since ${\rm div}~u=0$, we have
\begin{align}\label{5ineq1}
\Delta P = {\rm div}~{\rm div}~(\tau-u\otimes u),
\end{align}
which implies that
\begin{align}\label{5ineq2}
\nabla P = \nabla\Delta^{-1}{\rm div}~{\rm div}~(\tau-u\otimes u).
\end{align}
Applying $\Delta_j$ to $(\ref{eq2})_1$, we obtain
\begin{align}\label{5ineq3}
\frac{\partial}{\partial t}\Delta_ju + u\cdot\nabla\Delta_ju = -[\Delta_j,u\cdot\nabla]u + \Delta_j{\rm div}~\tau + \nabla(-\Delta)^{-1}{\rm div}~{\rm div}~(\tau-u\otimes u).
\end{align}
Integrating $(\ref{5ineq3})$ over $[0,T]$, we infer that
\begin{align}\label{5ineq4}
\|\Delta_ju\|_{L^{\infty}} &\leq \|\Delta_ju_0\|_{L^{\infty}}+ \int_0^T \|[\Delta_j,u\cdot\nabla]u\|_{L^{\infty}}dt \\ \notag
&~~~+\int_0^T C2^j\|\Delta_j\tau\|_{L^{\infty}} + \|\nabla(-\Delta)^{-1}{\rm div}~{\rm div}~(\tau-u\otimes u)\|_{L^{\infty}}dt.
\end{align}
By virtue of Lemma \ref{Lemma5}, we obtain
\begin{align}\label{5ineq5}
\|(Id-\Delta_{-1})\nabla(-\Delta)^{-1}{\rm div}~{\rm div}~\tau\|_{B^1_{\infty,1}}\leq C\|\tau\|_{B^2_{\infty,1}},
\end{align}
and
\begin{align}\label{5ineq5'}
\|\Delta_{-1}\nabla(-\Delta)^{-1}{\rm div}~{\rm div}~\tau\|_{B^1_{\infty,1}}
&\leq C\|\Delta_{-1}\nabla(-\Delta)^{-1}{\rm div}~{\rm div}~\tau\|_{L^\infty} \\ \notag
&\leq C\|\theta E_2\ast\Delta_{-1}\nabla{\rm div}~{\rm div}~\tau\|_{L^\infty}\\ \notag
&~~~+C\|\nabla{\rm div}~{\rm div}~(1-\theta) E_2\ast\Delta_{-1}\tau\|_{L^\infty}\\ \notag
&\leq C\|\Delta_{-1}\tau\|_{L^\infty}\leq C\|\tau\|_{B^2_{\infty,1}},
\end{align}
where $E_2=\ln |x|$ denotes the fundamental solution of the harmonic equation and $\theta$ represents a smooth function supported on the unit sphere. The estimate for the other term in the pressure is similar. Hence, we deduce from \eqref{5ineq4}-\eqref{5ineq5'} that
\begin{align}\label{5ineq6}
\|u\|_{L^{\infty}_T(B^1_{\infty,1})} \leq \|u_0\|_{B^1_{\infty,1}} + C(T\|u\|^2_{L^{\infty}_T(B^1_{\infty,1})}  + \int_0^T \|\tau\|_{B^2_{\infty,1}} dt).
\end{align}
Applying $\Delta_j$ to $(\ref{eq2})_2$, we obtain
\begin{align}\label{5ineq7}
\Delta_j\tau = e^{t\mu\Delta-at}\Delta_j\tau_0 + \int_0^t e^{(t-s)\mu\Delta-a(t-s)}(\Delta_j(Q(\Omega,\tau))+{\rm div}~\Delta_j(u\otimes\tau))ds.
\end{align}
By virtue of Lemmas \ref{Lemma1} and \ref{Lemma1'}, we infer that
\begin{align}\label{5ineq10}
\|\tau\|_{L^{\infty}_T(B^0_{\infty,1})}+ \|\tau\|_{L^{1}_T(B^2_{\infty,1})} &\leq \|\tau_0\|_{B^0_{\infty,1}} + \int_0^T \|Q(\Omega,\tau)\|_{B^0_{\infty,1}}dt+\int_0^T \|u\otimes\tau\|_{B^1_{\infty,1}}dt\\ \notag
&\leq \|\tau_0\|_{B^0_{\infty,1}} + C\|u\|_{L^{\infty}_T(B^1_{\infty,1})}T^{\frac{1}{2}}\|\tau\|_{L^2_T(B^1_{\infty,1})}.
\end{align}
Notice that
\begin{align}\label{5ineq14}
\|\tau\|_{L^2_T(B^1_{\infty,1})} \leq \|\tau\|^{\frac{1}{2}}_{L^{\infty}_T(B^0_{\infty,1})}\|\tau\|^{\frac{1}{2}}_{L^1_T(B^2_{\infty,1})}.
\end{align}
According to \eqref{5ineq6}, \eqref{5ineq10} and \eqref{5ineq14}, we obtain
\begin{align}\label{5ineq17}
 \|u\|_{L^{\infty}_T(B^1_{\infty,1})}+\|\tau\|_{L^{\infty}_T(B^0_{\infty,1})} + \|\tau\|_{L^1_T(B^2_{\infty,1})}
 &\leq C(\|u_0\|_{B^1_{\infty,1}} + \|\tau_0\|_{B^0_{\infty,1}}) + CT\|u\|^2_{L^{\infty}_TB^1_{\infty,1}} \\ \notag
 &~~~+CT\|u\|^2_{L^{\infty}_T(B^1_{\infty,1})}\|\tau\|_{L^{\infty}_T(B^0_{\infty,1})}.
\end{align}
Suppose that
\begin{align}\label{5ineq18}
\|u\|_{L^{\infty}_T(B^1_{\infty,1})}+\|\tau\|_{L^{\infty}_T(B^0_{\infty,1})} + \|\tau\|_{L^1_T(B^2_{\infty,1})} \leq 4C(\|u_0\|_{B^1_{\infty,1}} + \|\tau_0\|_{B^0_{\infty,1}}),
\end{align}
and
\begin{align}\label{5ineq19}
T=\min\{1,\frac{1}{16C^2(\|u_0\|_{B^1_{\infty,1}} + \|\tau_0\|_{B^0_{\infty,1}})},\frac{1}{64C^3(\|u_0\|_{B^1_{\infty,1}} + \|\tau_0\|_{B^0_{\infty,1}})^2}\}.
\end{align}
Plugging \eqref{5ineq18} and \eqref{5ineq19} into \eqref{5ineq17} leads to
\begin{align}
\|u\|_{L^{\infty}_T(B^1_{\infty,1})}+\|\tau\|_{L^{\infty}_T(B^0_{\infty,1})} + \|\tau\|_{L^1_T(B^2_{\infty,1})} \leq 3C(\|u_0\|_{B^1_{\infty,1}} + \|\tau_0\|_{B^0_{\infty,1}}).
\end{align}
We thus complete the proof of Proposition \ref{5prop1}.
\end{proof}
\begin{prop}\label{BKM2}
Assume that $d=2$. Let $(u,\tau)$ be a strong solution of \eqref{eq2} with the initial data $(u_0,\tau_0)\in (H^1\cap B^1_{\infty,1})\times(H^1\cap B^0_{\infty,1})$. Suppose that $T^\ast$ is the maximal existence time, then the solution blows up in finite time $T^\ast<\infty$ if and only if
\begin{align}
\int_0^{T^{\ast}}\|\Omega\|_{B^0_{\infty,1}}dt = \infty.
\end{align}
\end{prop}
\begin{proof}
According to Bony's decomposition, we obtain
\begin{align}\label{1BKM2}
\int_0^T \|Q(\Omega,\tau)\|_{B^0_{\infty,1}}dt \leq \int_0^T \|\Omega\|_{L^{\infty}}\|\tau\|_{B^0_{\infty,1}}+(\|\tau\|_{L^{\infty}}+\|\tau\|_{H^1})\|u\|_{B^1_{\infty,1}}dt.
\end{align}
By virtue of Lemma \ref{transportdiffusion} and Proposition \ref{7lemm1}, we deduce from \eqref{eq2} and \eqref{1BKM2} that
\begin{align}\label{2BKM2}
&\|\tau\|_{L^{\infty}_T(B^{0}_{\infty,1})}+\|\tau\|_{L^1_T
	(B^2_{\infty,1})}\\ \notag
&\leq Ce^{CT+C\int_0^T\|\Omega\|_{B^0_{\infty,1}}dt}\Big(\|\tau_0\|_{B^0_{\infty,1}}
+ \int_0^T\|Q(\Omega,\tau)\|_{B^{0}_{\infty,1}}dt\Big)\\ \notag
&\leq Ce^{CT+C\int_0^T\|\Omega\|_{B^0_{\infty,1}}dt}\Big(\|\tau_0\|_{B^0_{\infty,1}}
+ \int_0^T \|u\|_{B^1_{\infty,1}}+\|\Omega\|_{L^{\infty}}\|\tau\|_{B^0_{\infty,1}}dt\Big).
\end{align}
According to Lemma \ref{transportdiffusion} and \eqref{2BKM2}, we infer that
\begin{align}\label{3BKM2}
&\|u\|_{L^{\infty}_T(B^1_{\infty,1})}+\|\tau\|_{L^{\infty}_T(B^{0}_{\infty,1})}\\ \notag
&\leq Ce^{CT+C\int_0^T\|\Omega\|_{B^0_{\infty,1}}dt}\Big(\|u_0\|_{B^1_{\infty,1}}+\|\tau_0\|_{B^0_{\infty,1}}
+ \int_0^T\|u\|_{B^1_{\infty,1}}+\|\Omega\|_{L^{\infty}}\|\tau\|_{B^0_{\infty,1}}dt\Big).
\end{align}
Applying Gronwall's inequality to \eqref{3BKM2}, we deduce that
\begin{align}\label{4BKM2}
\|u\|_{L^{\infty}_TB^1_{\infty,1}}+\|\tau\|_{L^{\infty}_TB^{0}_{\infty,1}}\leq C\Big(\|u_0\|_{B^1_{\infty,1}}+\|\tau_0\|_{B^0_{\infty,1}}\Big)e^{e^{CT+C\int_0^T\|\Omega\|_{B^0_{\infty,1}}dt}}.
\end{align}
Assume that $T^\ast<\infty$ and $\int_{0}^{T^\ast}\|\Omega(t)\|_{B^0_{\infty,1}}dt<\infty$. By virtue of Proposition \ref{5prop1} and \eqref{4BKM2}, we infer that the solution can be continued beyond $[0,T^{\ast})$, which contradicts the assumption that $T^\ast$ is the maximal existence time.
\end{proof}
\subsection{Global well-posedness}
{\bf The proof of Theorem \ref{th13} :}  \\
Notice that
\begin{align}
	\int_0^T\|\mu\Omega\|_{B^0_{\infty,1}}dt \leq \int_0^T\|\Gamma\|_{B^0_{\infty,1}}dt + \int_0^T\|\tau\|_{B^0_{\infty,1}}dt.
\end{align}
We infer from Proposition \ref{7lemm2} and Proposition \ref{BKM2} that the estimate of $\|\Gamma\|_{B^0_{\infty,1}}$ will finish the proof of  global existence for \eqref{eq2} in critical Besov space. Recall that
\begin{align}\label{7ineq13}
	\frac{d}{dt}\Gamma + u\cdot\triangledown\Gamma = aR\tau + RQ(\Omega,\tau) + [R,u\cdot\triangledown]\tau \triangleq \sum_{i = 1}^3 F_i,
\end{align}
Note that
\begin{align}\label{smallness1}
	E_0=H_0(\|\tau_0\|_{B^0_{\infty,1}}+\|\tau_0\|_{L^4})~\text{and}~D_0 =\|(\mu\nabla u_0,\tau_0)\|_{B^0_{\infty,1}}.
\end{align}
Suppose that for any $t\in[0,T)$, we have
\begin{align}\label{assumption1}
	\|\Gamma\|_{B^0_{\infty,1}}\leq  c_1 a\mu e^{\frac{a}{4}t},
\end{align}
for some $c_1$ small enough. Applying Lemma \ref{Lemma1''} to \eqref{7ineq13}, we obtain
\begin{align}\label{7ineq14}
	\|\Gamma\|_{B^{0}_{\infty,1}} \leq C(\|\Gamma_0\|_{B^{0}_{\infty,1}} + \sum_{i = 1}^3\int_0^t \|F_i\|_{B^{0}_{\infty,1}} ds)(1+\int_0^t\|\nabla u\|_{L^{\infty}}ds).
\end{align}
According to Lemma \ref{Lemma5}, Propositions \ref{7lemm1} and \ref{7lemm2}, we have
\begin{align}\label{7ineq15}
	\int_0^t\|\nabla u\|_{L^{\infty}}ds &\leq \int_0^t\|\Delta_{-1}\nabla u\|_{L^{\infty}}+\|(Id-\Delta_{-1})\nabla u\|_{L^{\infty}}ds \\  \notag
	&\leq C\int_0^t\|\nabla u\|_{L^2}ds+C\int_0^t\|\Omega\|_{B^0_{\infty,1}}ds\\ \notag
	&\leq CtH^{\frac 1 2}_0+C{\mu}^{-1}(\int_0^t\|\Gamma\|_{B^0_{\infty,1}}ds +\int_0^t\|\tau\|_{L^4}ds+ \int_0^t\|\tau\|_{B^0_{\infty,1}}ds)\\ \notag
	&\leq CtH^{\frac 1 2}_0+C\big(e^{\frac{a}{4}t} + a^{-\frac{1}{2}}{\mu}^{-1}B_0\big).
\end{align}
Using the conditions \eqref{highsmallness} and \eqref{nonlinearsmallness}, the we get
\begin{align}\label{7ineq16}
	\|\tau_0\|_{L^4}\leq c^2\min\{a^2, a\mu, \mu^{\frac 3 2} a^{\frac 1 2}\},~D_0\leq ca\mu~\text{and}~E_0\leq c^2\min\{a^2\mu^2, a\mu^3,a^3\mu\}.
\end{align}
Then we have
\begin{align}\label{7ineq17}
	B_0&=C(a^{-\frac{1}{2}}+\mu^{-\frac{1}{2}})\|\tau_0\|_{L^4}+ C\mu^{-1}a^{-\frac{1}{2}}H^{\frac 1 2}_0\|\tau_0\|_{L^{\infty}}  \\\notag
	&\leq ca^{\frac{1}{2}}\mu+C\mu^{-1}a^{-\frac{1}{2}}E^{\frac 1 2}_0 D^{\frac 1 2}_0
	\\\notag
	&\leq 2ca^{\frac{1}{2}}\mu.
\end{align}
By virtue of \eqref{7ineq15}-\eqref{7ineq17}, we obtain $1+\int_0^t\|\nabla u\|_{L^{\infty}}ds \leq CtH^{\frac 1 2}_0+Ce^{\frac{a}{4}t}$. This together with the condition \eqref{highsmallness} implies that
\begin{align}\label{7ineq18}
	C\|\Gamma_0\|_{B^{0}_{\infty,1}}(1+\int_0^t\|\nabla u\|_{L^{\infty}}ds)
	\leq C(tH^{\frac 1 2}_0+e^{\frac{a}{4}t})(D_0+\|\tau_0\|_{L^4})
	\leq \frac {c_1}{10}a\mu e^{\frac{a}{4}t}.
\end{align}
According to Lemma \ref{Lemma5} and Proposition \ref{7lemm2}, we deduce that
\begin{align}\label{7ineq19}
	\int_0^t \|F_1\|_{B^{0}_{\infty,1}} ds &\leq \int_0^t a\|\Delta_{-1}R\tau\|_{L^{\infty}} + a\|(Id - \Delta_{-1})R\tau\|_{B^{0}_{\infty,1}} ds \\ \notag
	&\leq C\int_0^t a\|\tau\|_{L^4} ds+C\int_0^t a\|\tau\|_{B^{0}_{\infty,1}} ds\\ \notag
	&\leq C\|\tau_0\|_{L^4}+ Ca(\int_0^te^{as}\|\tau\|^2_{B^{0}_{\infty,1}}ds)^{\frac{1}{2}}(\int_0^te^{-as}ds)^{\frac{1}{2}}\\ \notag
	&\leq C\|\tau_0\|_{L^4}+Ca^{\frac{1}{2}}B_0.
\end{align}
Thus we infer from \eqref{7ineq16}, \eqref{7ineq17} and \eqref{7ineq19} that
\begin{align}\label{7ineq20}
	C\int_0^t \|F_1\|_{B^{0}_{\infty,1}} ds(1+\int_0^t\|\nabla u\|_{L^{\infty}}ds) &\leq C(\|\tau_0\|_{L^4}+a^{\frac{1}{2}}B_0)(tH^{\frac 1 2}_0+e^{\frac{a}{4}t}) \\ \notag
	&\leq \frac{c_1}{20}a\mu e^{\frac{a}{4}t}+C(\|\tau_0\|_{L^4}+a^{\frac{1}{2}}B_0)tH^{\frac 1 2}_0 \\ \notag
	&\leq \frac{c_1}{20}a\mu e^{\frac{a}{4}t}+\frac C a e^{\frac a 4 t}(\|\tau_0\|^{\frac 1 2}_{L^4}E^{\frac 1 2}_0+a^{\frac 1 2}\mu^{-\frac 1 2}\|\tau_0\|^{\frac 1 2}_{L^4}E^{\frac 1 2}_0) \\ \notag
	&~~~+\frac C a e^{\frac a 4 t}\mu^{-1}E_0\\ \notag
	&\leq \frac{c_1}{10}a\mu e^{\frac{a}{4}t}.
\end{align}
By virtue of Lemma \ref{Lemma5}, Propositions \ref{7lemm1} and \ref{7lemm2}, we obtain
\begin{align}\label{7ineq21}
	\int_0^t \|F_2\|_{B^{0}_{\infty,1}} ds
	&\leq C\int_0^t \|Q(\Omega,\tau)\|_{L^{2}}ds+C\int_0^t \|Q(\Omega,\tau)\|_{B^{0}_{\infty,1}}ds \\ \notag
	&\leq \frac C a H^{\frac 1 2}_0\|\tau_0\|_{L^\infty}+C\int_0^t \|\Omega\|_{B^{0}_{\infty,1}}\|\tau\|_{B^{\frac{1}{4}}_{\infty,1}}ds \\ \notag
	&\leq  \frac C a H^{\frac 1 2}_0\|\tau_0\|_{L^\infty}+ C\mu^{-1}\int_0^t\|\Gamma\|_{B^{0}_{\infty,1}}\|\tau\|_{B^{\frac{1}{4}}_{\infty,1}}+\|R\tau\|_{B^{0}_{\infty,1}}\|\tau\|_{B^{\frac{1}{4}}_{\infty,1}}ds\\ \notag
	&\leq  \frac C a H^{\frac 1 2}_0\|\tau_0\|_{L^\infty}+C(a^{\frac{1}{2}}B_0+\mu^{-1}B_0^2).
\end{align}
Then we deduce from \eqref{7ineq16}, \eqref{7ineq17} and \eqref{7ineq21} that
\begin{align}\label{7ineq22}
	C\int_0^t \|F_2\|_{B^{0}_{\infty,1}} ds
	(1+\int_0^t\|\nabla u\|_{L^{\infty}}ds)&\leq C(\frac 1 a H^{\frac 1 2}_0\|\tau_0\|_{L^\infty}+a^{\frac{1}{2}}B_0+\mu^{-1}B_0^2)(tH^{\frac 1 2}_0+e^{\frac{a}{4}t}) \\ \notag
	&\leq \frac{c_1}{20}a\mu e^{\frac{a}{4}t}+ \frac C a H^{\frac 1 2}_0\|\tau_0\|_{L^\infty}(tH^{\frac 1 2}_0+e^{\frac{a}{4}t}) \\ \notag
	&\leq \frac{c_1}{10}a\mu e^{\frac{a}{4}t}.
\end{align}
We infer from Lemma \ref{Lemma2} that
\begin{align}\label{7ineq23}
	\|[R,u\cdot\nabla]\tau\|_{B^{0}_{\infty,1}} &\leq C(\|\Omega\|_{L^{\infty}}+\|\Omega\|_{L^4})(\|\tau\|_{B^{\frac{1}{4}}_{\infty,1}} + \|\tau\|_{L^4})\\ \notag
	&\leq C(\|\Omega\|_{L^{\infty}}+\|u\|_{H^1})(\|\tau\|_{B^{\frac{1}{4}}_{\infty,1}} + \|\tau\|_{L^4}) \\ \notag
	&\leq C(\mu^{-1}\|\Gamma\|_{L^{\infty}}+\mu^{-1}\|\tau\|_{B^{\frac{1}{4}}_{\infty,1}}+\mu^{-1}\|\tau\|_{L^4}+\|u\|_{H^1})(\|\tau\|_{B^{\frac{1}{4}}_{\infty,1}} + \|\tau\|_{L^4}).
\end{align}
According to \eqref{7ineq23}, Propositions \ref{7lemm2} and \ref{7lemm1'}, we get
\begin{align}\label{7ineq24}
	\int_0^t \|F_3\|_{B^{0}_{\infty,1}} ds
	&\leq  \int_0^t C(\mu^{-1}\|\Gamma\|_{L^{\infty}}+\mu^{-1}\|\tau\|_{B^{\frac{1}{4}}_{\infty,1}}+\mu^{-1}\|\tau\|_{L^4}+\|u\|_{H^1})(\|\tau\|_{B^{\frac{1}{4}}_{\infty,1}} + \|\tau\|_{L^4})ds \\ \notag
	&\leq  C(a^{\frac{1}{2}}B_0+\mu^{-1}B_0^2+\|\tau_0\|_{L^4}+a^{-\frac{1}{2}}\mu^{-1}\|\tau_0\|_{L^4}B_0 \\ \notag
	&~~~+a^{-\frac{1}{2}}H^{\frac 1 2}_0B_0+a^{-1}H^{\frac 1 2}_0\|\tau_0\|_{L^4}).
\end{align}
Using the conditions \eqref{nonlinearsmallness}, we have $\|\tau_0\|_{L^4}\leq c^2\min\{a\mu,a^{2}\}$.
Then we deduce from \eqref{7ineq24} that
\begin{align}\label{7ineq25}
	C\int_0^t \|F_3\|_{B^{0}_{\infty,1}} ds(1+\int_0^t\|\nabla u\|_{L^{\infty}}ds) &\leq  \frac{c_1}{30}a\mu e^{\frac{a}{4}t}+C(a^{-\frac{1}{2}}H^{\frac 1 2}_0B_0+a^{-1}H^{\frac 1 2}_0\|\tau_0\|_{L^4})e^{\frac{a}{4}t}\\ \notag
	&~~~+C(a^{-\frac{1}{2}}H^{\frac 1 2}_0B_0+a^{-1}H^{\frac 1 2}_0\|\tau_0\|_{L^4})tH^{\frac 1 2}_0\\ \notag
	&\leq \frac{c_1}{20}a\mu e^{\frac{a}{4}t}+C(a^{-1}\mu^{-1}E_0+E^{\frac 1 2}_0)e^{\frac{a}{4}t}\\ \notag
	&\leq \frac{c_1}{10}a\mu e^{\frac{a}{4}t},
\end{align}
where we use the conditions $E_0\leq c^2\min\{a^2\mu^2,a^{\frac 5 2}\mu^{\frac 3 2}\}$ and $H^{\frac 3 2}_0\|\tau_0\|_{L^\infty}\leq c^2a^3\mu^2$.
Combining above estimates for \eqref{7ineq14}, we infer that
\begin{align}\label{7ineq26}
	\|\Gamma\|_{B^{0}_{\infty,1}}\leq \frac{c_1}{2}a\mu e^{\frac{a}{4}t},
\end{align}
which implies that $T^\ast=+\infty$.
We thus complete the proof of Theorem \ref{th13}.
\hfill$\Box$
\section{Large time behavior for the noncorotation inviscid Oldroyd-B model}
In this section we consider large time behavior of global solutions in $H^1$ for \eqref{eq0} with $\nu=0$. For simplify, the other parameters in \eqref{eq0} will be taken as the constant 1.

For the reader's convenience, we first recall the following theorem.
\begin{theo}\cite{2015Elgindi}\label{th3}
Let $d=2~and~s>2$. Assume that $a=\mu=1$. Let $(u,\tau)$ be a strong solution of \eqref{eq0} with $\nu=0$ and the initial data $(u_0,\tau_0)\in H^s$. Then, there exists some sufficiently small constant $\delta$ such that if
\begin{align}
\|(u_0,\tau_0)\|_{H^1}+\|(\Omega_0,\tau_0)\|_{B^0_{\infty,1}}\leq \delta,~~~~\Omega_0={\rm curl}~u_0,
\end{align}
then the system \eqref{eq0} admits a unique global strong solution $(u,\tau)\in C([0,\infty); H^s)$. Moreover, the energy estimation for $(u,\tau,\Gamma)$ with $\Gamma=\Omega-R\tau$ implies that
\begin{align}\label{estimate}
\frac{d}{dt}\|(u,\tau)\|^2_{H^1}  + \|\nabla u\|^2_{L^2} + \|\tau\|^2_{H^2} \leq 0.
\end{align}
\end{theo}

Similar to \cite{He2009} and \cite{Luo-Yin}, we can cancel ${\rm div}~\tau$ in Fourier space and prove the following initial time decay rate of $(u,\tau)$ in $H^1$ by the Fourier splitting method and the bootstrap argument.
\begin{prop}\label{prop3}
Under the condition in Theorem $\ref{th2}$. Then there exists $C>0$ such that for any $l\in N$ and $t>0$, we have
\begin{align}\label{decay1}
\|(u,\tau)\|_{H^1} \leq C\ln^{-l}(e+t).
\end{align}
\end{prop}
\begin{proof}
Let $S_0(t) =\{\xi:f(t)|\xi|^2\leq 2C_2f'(t)\}$ with $C_2$ large enough. According to Theorem \ref{th3}, we have
\begin{align}\label{4ineq1}
\frac{d}{dt}[f(t)\|(u,\tau)\|^2_{H^1}]+C_2f'(t)\|u\|^2_{L^2} + f(t)\|\tau\|^2_{H^2}\leq Cf'(t)\int_{S_0(t)}|\hat{u}|^2 d\xi+f'(t)\|\nabla u\|^2_{L^2},
\end{align}
for some $t>0$ sufficiently large. Applying Fourier transformation to \eqref{eq0}, we obtain
\begin{align}\label{4eq1}
\left\{\begin{array}{l}
\frac{d}{dt}\hat{u} + i\xi^{T}\mathscr{F}(u \otimes u) + i\xi\hat{p} = i\xi^{T}\hat{\tau},\\
\frac{d}{dt}\hat{\tau} + \hat{\tau} +\mathscr{F}(u \cdot\nabla\tau) + |\xi|^2\hat{\tau} +\mathscr{F}Q(\nabla u,\tau)= \frac{i}{2}( \xi \otimes \hat{u} + \hat{u} \otimes \xi ).
\end{array}\right.
\end{align}
Multiplying $\eqref{4eq1}$ by $(\bar{\hat{u}},\bar{\hat{\tau}})$ and taking the real part, we deduce that
\begin{align}
\frac{1}{2}\frac{d}{dt}|\hat{u}|^2  = \mathcal{R}e[-i\xi^{T}\mathscr{F}(u \otimes u)\bar{\hat{u}}+i\xi^{T}\hat{\tau}\bar{\hat{u}}],
\end{align}
and
\begin{align}
\frac{1}{2}\frac{d}{dt}|\hat{\tau}|^2 + |\hat{\tau}|^2 + |\xi|^2|\hat{\tau}|^2=\mathcal{R}e[\mathscr{F}(u \cdot\nabla\tau):\bar{\hat{\tau}}-\mathscr{F}Q(\nabla u,\tau):\bar{\hat{\tau}}+\frac{i}{2}(\xi \otimes \hat{u} + \hat{u} \otimes \xi ):\bar{\hat{\tau}}].
\end{align}
Since $\tau$ is symmetric, we have
\begin{align}
\mathcal{R}e[i\xi^{T}\hat{\tau}\bar{\hat{u}} + \frac{i}{2}( \xi \otimes \hat{u} + \hat{u} \otimes \xi ):\bar{\hat{\tau}}] = 0,
\end{align}
which implies that
\begin{align}\label{4ineq2}
\frac{1}{2}\frac{d}{dt}(|\hat{u}|^2 + |\hat{\tau}|^2 ) + |\hat{\tau}|^2 + |\xi|^2|\hat{\tau}|^2 &= \mathcal{R}e[- i\xi^{T}\mathscr{F}(u \otimes u)\bar{\hat{u}} -\mathscr{F}(u \cdot\nabla\tau):\bar{\hat{\tau}} -\mathscr{F}Q(\nabla u,\tau):\bar{\hat{\tau}}] \\ \notag
&\leq |\xi||\mathscr{F}(u \otimes u)||\hat{u}|+|\mathscr{F}(u \cdot\nabla\tau)|^2 + |\mathscr{F}Q(\nabla u,\tau)|^2+|\hat{\tau}|^2.
\end{align}
Let $f(t)=\ln^3(e+t)$. According to Theorem \ref{th3}, we have
\begin{align*}
\int_{S_0(t)}|\hat{u}|^2+|\hat{\tau}|^2 d\xi&\leq\|(u_0,\tau_0)\|^2_{L^1}\frac{f'(t)}{f(t)} + \int_0^t \int_{S_0(t)}|\xi||\mathscr{F}(u\otimes u)||\hat{u}| d\xi ds \\
&~~~+\int_0^t \int_{S_0(t)}|\mathscr{F}(u \cdot\nabla\tau)|^2+|\mathscr{F}Q(\nabla u,\tau)|^2 d\xi ds \\
&\leq C\frac{f'(t)}{f(t)} + \int_0^t \|u\|^3_{L^2} (\int_{S_0(t)}|\xi|^2d\xi)^{\frac{1}{2}} ds  \\
&~~~+ \frac{f'(t)}{f(t)}\int_0^t \|\nabla u\|^2_{L^2}\|\tau\|^2_{L^2} + \| u\|^2_{L^2}\|\nabla\tau\|^2_{L^2}ds \\
&\leq C\frac{f'(t)}{f(t)} +C\frac{f'(t)}{f(t)}\int_0^t \|u\|^3_{L^2} ds  \\
&\leq C\frac{f'(t)}{f(t)}+C\frac{f'(t)}{f(t)}(1+t)  \\
&\leq C\ln^{-1}( e + t ).
\end{align*}
This together with \eqref{4ineq1} and \eqref{estimate} ensures that
\begin{align*}
f(t)\|(u,\tau)\|^2_{H^1} &\leq C+C\int_0^t f'(s)\ln^{-1}( e + t )ds + C\int_0^t f'(s)\|\nabla u\|^2_{L^2} ds \\
&\leq C\ln^2( e + t ),
\end{align*}
which implies that
\begin{align}\label{4ineq3}
\|(u,\tau)\|^2_{H^1}  \leq C\ln^{-1}(e + t).
\end{align}
We prove \eqref{decay1} by induction. Assume that
\begin{align}\label{4ineq4}
\|(u,\tau)\|^2_{H^1}\leq \ln^{-l} (e + t).
\end{align}
Let $f(t)=\ln^{l+3}(e+t)$. Using \eqref{4ineq4}, we can deduce that
\begin{align*}
\int_{S_0(t)}|\hat{u}|^2 + |\hat{\tau}|^2 d\xi&\leq C\frac{f'(t)}{f(t)} + \frac{f'(t)}{f(t)}\int_0^t \|u\|^3_{L^2} ds  \\
&\leq C\ln^{-\frac {3l} {2}-1}(e+t).
\end{align*}
This together with \eqref{4ineq1} and \eqref{estimate} ensures that
\begin{align*}
f(t)\|(u,\tau)\|^2_{H^1} &\leq C+C\int_0^t f'(s)\ln^{-\frac {3l} {2}-1}( e + t )ds + C\int_0^t f'(s)\|\nabla u\|^2_{L^2} ds \\
& \leq C\ln^2( e + t ),
\end{align*}
which implies that
\begin{align*}
\|(u,\tau)\|^2_{H^1} \leq C\ln^{-l-1} (e + t).
\end{align*}
We thus complete the proof of Proposition \ref{prop3}.
\end{proof}

{\bf The proof of Theorem \ref{th2} :}  \\
Now we are going to improve initial time decay rate in Proposition \ref{prop3}. Let $S(t) = \{ \xi | |\xi|^2 \leq C_2(1+t)^{-1} \}$ with sufficiently large $C_2> 0$ and $t>0$. According to Theorem \ref{th3}, we obtain
\begin{align*}
&\frac{d}{dt}(\|(u,\tau)\|^2_{H^1} )+C_2(1+t)^{-1}\|u\|^2_{L^2}+\|\tau\|^2_{H^2}\leq C(1+t)^{-1}\int_{S(t)}|\hat{u}|^2d\xi,
\end{align*}
which implies that
\begin{align}\label{4ineq5}
&\frac{d}{dt}[(1+t)^2\|(u,\tau)\|^2_{H^1}] + \frac 1 2 C_2(1+t)\|u\|^2_{L^2}+\frac 1 2(1+t)^2\|\tau\|^2_{H^2} \\ \notag
&\leq C(1+t)\int_{S(t)}|\hat{u}|^2d\xi+ C(1+t)\|\nabla u\|^2_{L^2}.
\end{align}
Integrating \eqref{4ineq2} over $S(t)\times[0,t]$ with $(\xi,s)$ and according to Theorem \ref{th3}, we can deduce that
\begin{align}\label{4ineq6}
\int_{S(t)}|\hat{u}|^2+|\hat{\tau}|^2 d\xi&\leq \frac{C}{1+t} + \int_0^t \int_{S(t)}|\xi||\mathscr{F}(u\otimes u)||\hat{u}|+|\mathscr{F}(u \cdot\nabla\tau)|^2+|\mathscr{F}Q(\nabla u,\tau)|^2 d\xi ds \\  \notag
&\leq \frac{C}{1+t} + \frac{C}{1+t}\int_0^t \|u\|^3_{L^2} ds.
\end{align}
Together with \eqref{4ineq5} and \eqref{estimate}, we infer that
\begin{align}\label{4ineq7}
(1+t)^2\|(u,\tau)\|^2_{H^1}&\leq \|(u_0,\tau_0)\|^2_{H^1} + C(1+t)+ C\int_0^t\int_0^s \|u\|^3_{L^2} ds' ds + C\int_0^t (1+s)\|\nabla u\|^2_{L^2}ds \\ \notag
& \leq  C(1+t)+ C(1+t)\int_0^t\|u\|^3_{L^2} ds + C\int_0^t \|(u,\tau)\|^2_{H^1}ds \\  \notag
& \leq  C(1+t)+ C(1+t)\int_0^t\|u\|^3_{L^2} ds.
\end{align}
Let $M(t) = \mathop{\sup}\limits_{s \in [0,t]} (1+s)\|(u,\tau)\|^2_{H^1}$. Using \eqref{4ineq7} and \eqref{decay1} with $l=2$, we obtain
\begin{align}\label{4ineq8}
& M(t) \leq C+C\int_0^t M(s)(1+s)^{-1}\ln^{-2}(e+t) ds,
\end{align}
Applying Gronwall's inequality to \eqref{4ineq8}, we get
\begin{align*}
 M(t) \leq C e^{C\int_0^t (1+s)^{-1}\ln^{-2}(e+t) ds}\leq C,
\end{align*}
which implies that
\begin{align*}
\|u\|^2_{H^1} + \|\tau\|^2_{H^1} \leq C(1 + t)^{-1}.
\end{align*}
We thus complete the proof of Theorem \ref{th2}.
\hfill$\Box$

\smallskip
\noindent\textbf{Acknowledgments} This work was partially supported by the National Natural Science Foundation of China (No.12171493), the Macao Science and Technology Development Fund (No. 0091/2018/A3), and Guangdong Province of China Special Support Program (No. 8-2015).


\phantomsection
\addcontentsline{toc}{section}{\refname}
\bibliographystyle{abbrv} 
\bibliography{OldroydBref}

\end{document}